%% file: TST-1.tex
\documentclass[a4paper,12pt,oneside,reqno]{amsart}
\usepackage{amssymb,amsfonts,amsmath,amsthm,amscd}
\pdfoutput=1
\usepackage{graphicx, color}
\newcommand{\beq}{\begin{equation}}
\newcommand{\eeq}{\end{equation}}
\newcommand{\bea}{\begin{aligned}}
\newcommand{\eea}{\end{aligned}}
\newcommand{\bdm}{\begin{displaymath}}
\newcommand{\edm}{\end{displaymath}}
\newcommand{\barr}{\begin{array}}
\newcommand{\earr}{\end{array}}
\newcommand{\ben}{\begin{enumerate}}
\newcommand{\een}{\end{enumerate}}
\newcommand{\bde}{\begin{description}}
\newcommand{\ede}{\end{description}}

\numberwithin{equation}{section}
\newtheorem{teor}{Theorem}[section]
\newtheorem{prop}[teor]{Proposition}
\newtheorem{lem}[teor]{Lemma}

\newtheorem{Def}[teor]{Definition}

\newtheorem{rem}[teor]{Remark}

\newcommand{\R}{\mathbb{R}}

\newcommand{\N}{\mathbb{N}}

\newcommand{\PP}{\mathbb{P}}
\newcommand{\E}{\mathbb{E}}

\newcommand{\al}{\alpha}

\newcommand{\de}{\delta}

\newcommand{\e}{\epsilon}

\newcommand{\X}{\mathcal{X}}
\newcommand{\M}{\mathcal{M}}

\begin{document}
\title[Trait substitution tree]{Trait substitution trees on two time scales analysis}
\author{Anton Bovier and Shi-Dong Wang}

\address{A. Bovier\\Institut f\"ur Angewandte Mathematik\\Rheinische
   Friedrich-Wilhelms-Uni\-ver\-si\-t\"at Bonn\\Endenicher Allee 60\\ 53115
   Bonn, Germany}
\email{bovier@uni-bonn.de}
\address{S.-D. Wang\\Department of Statistics\\University of Oxford
\\1 South Parks Road
\\ Oxford, OX1 3TG, UK}
\email{shidong.wang@stats.ox.ac.uk}


\thanks{A. Bovier is supported by in part by the DFG in SPP 1590 ``Probabilistic Structures in Evolution'' and the 
Hausdorff Center for Mathematics. S.-D. Wang was supported by a 
Hausdorff Scholarship while at the University of Bonn, and EPSRC Grant EP/I01361X/1 while at the University of Oxford.}

\date{\today}

\begin{abstract}

In this paper we consider two continuous-mass population models as analogues of logistic branching random walks, one is supported on a finite trait space and the other one is supported on an infinite trait space. For the first model with nearest-neighbor competition and migration, we justify a well-described evolutionary path to the short-term equilibrium on a slow migration time scale. For the second one with an additional evolutionary mechanism-mutation, a jump process-trait substitution tree model is established under a combination of rare mutation and slow migration limits. The transition rule of the tree highly depends on the relabeled trait sequence determined by the fitness landscape. The novelty of our
model is that each trait, which may nearly die out on the migration time scale,
has a chance to recover and further to be stabilized on the mutation time scale because
of a change in the fitness landscape due to a newly entering mutant.

\end{abstract}

\maketitle

\tableofcontents


\section{Introduction}\label{section 2.1}

In recent years a spatially structured population with migration (namely mutation in \cite{Cha06}) and local regulation proposed by Bolker and Pacala \cite{BP97}, Dieckmann and Law \cite{DL96} (BPDL process) has attracted particular interest both from biologists and mathematicians. It has several advantages over the traditional branching processes, which make it more natural as a population model: the quadratic competition term is used to prevent the population size from escaping to infinity, and  the mutation term is used to create an alternative trait type of the population for selection. Over the last decade, a lot 
of work has been  addressing different aspects on this model. 
For instance, Etheridge \cite{Eth04}, Fournier and M\'el\'eard \cite{FM04}, and Hutzenthaler and Wakolbinger \cite{HW07} study the extinction and survival problems. Champagnat \cite{Cha06}, Champagnat and Lambert \cite{CL07}, Champagnat and M\'el\'eard \cite{CM10}, M\'el\'eard and Tran \cite{MT09}, 
Dawson and Greven \cite{DG10} 
mainly focus on its long time behavior by multi-scale analysis methods.

The present work is largely motivated by the derivation of macroscopic 
phenomena on the level of populations from the individual based models
in the joint limits of large population size and small mutation rates.
We mention in particular the work of Fournier and M\'el\'eard \cite{FM04}, 
Champagnat \cite{Cha06} where under certain conditions convergence
to the so-called ``trait substitution sequence (TSS)'' was obtained. 
More recently, this type of results was extended in Champagnat and M\'el\'eard
\cite{CM10} to include further evolutionary phenomena such as evolutionary 
branching. A common feature of these works is the following setup: 
one assumes that mutations rates are so small that a monomorphic population, after a single mutation event has sufficient time to move to a new 
equilibrium where either the mutant trait gets extinct or the mutant trait 
fixated and the resident trait gets extinct. In this way, one obtains, on
the time scale at which such rare mutations occur, a sequence of populations 
evolving towards increasing fitness, the so-called \emph{trait substitution sequence}. 
In certain singular situations, one may also reach an equilibrium with co-existing traits, leading to the above-mentioned phenomenon of evolutionary branching \cite{CM10}.

What we wish to add to this picture in the present paper is a more complex 
structure of populations. The general idea is to consider populations with
 multiple traits where individual may change (upon birth or otherwise) 
between a \emph{finite} set of traits at a given (population size independent) rate. We will term such switches ``migrations''. In addition, there are
\emph{rare} mutations where an individual can be born with a \emph{new} trait
which has never been existing in the population.

This set-up is motivated from ideas that are currently discussed intensely in 
cancer research. The \emph{migration} events can be interpreted as epigenetic 
switch in the gene-expression of a cell between a variety of possible 
 ``metastable'' state (see e.g. Huang \cite{Huang11} and Gillies et al 
\cite{GVG12}  and references therein). Mutations are then true mutations that
 lead to a change in the epigenetically accessible trait-space. See also H\"olzel et al for a discussion in the
context of cancer evolution \cite{holzel2013}. 
In this paper we consider a very simple caricature of such a complex 
situations. Our purpose here is limited to showing that such models are still 
accessible to the mathematical methods developed in recent years, and that
such systems give rise to new and interesting mathematical structures.

In this paper  we investigate the long term behavior in a two-step limiting 
procedure where we first let the population size tend to infinity, and then let the migration rate tend to zero while rescaling time in an appropriate way
to obtain a non-trivial limit.
For a finite trait space, specific conditions are imposed on the fitness and 
demographic parameters, and a well-described evolutionary path to approach 
the short-term equilibrium will be obtained on an appropriate time scale.
The noteworthy feature here is that these equilibria can be polymorphic. 
We call this process  a \emph{trait substitution tree (TST)} on the finite trait 
space. For any given sequence of traits, the equilibrium configuration is determined by their labeled order according to their fitness landscape.

In a second step, we add random mutations on a longer time scale. 
This is modeled here as the appearance of mass at hitherto unoccupied 
locations in trait space driven by some Poisson process. The effect of the 
appearance of such new mass is a reshuffling of the migration part of the process that ends in a new equilibrium configuration.
As this process continues, we obtain what we call the \emph{trait substitution tree (TST)} process on 
infinite state space. 
The somewhat artificial introduction for mutations in the infinite population 
model is motivated on the basis of a limit of a finite population model with 
migration and mutation rates at distinct time scales. Such a model is studied 
in a companion paper \cite{BW12b}.

The remainder of the paper is organized as follows.
In Section \ref{section 2.2}, we briefly describe the microscopic model and give some preliminary results. In particular, we recall the law of large numbers of the BPDL processes.
In Section \ref{section 2.3}, as $\epsilon$ tends to 0, on a finite trait space we retrieve a well-defined short-term evolution path to its TST configuration on the migration time scale $O\left(\ln\frac{1}{\epsilon}\right)$.
In Section \ref{section 2.4}, under the rare mutation constraint we obtain a jump-type TST process on a longer time scale-the mutation time scale.
In Section \ref{section 2.5}, we provide proofs of the results in Section \ref{section 2.3} and Section \ref{section 2.4}.
Finally, for better understanding the TST process we provide a simulation algorithm in Section
 \ref{section 2.6}.


\section{Microscopic model}\label{section 2.2}
\subsection{Notation and description of the processes}
Following \cite{BP97}, we assume the population at time $t$ is composed of a finite number $I(t)$ of individuals characterized by their phenotypic traits $x_1(t), \cdots, x_{I(t)}(t)$ taking values (which can be equal) in a compact subset $\X$ of $\R^d$.

We denote by $\M_F(\X)$ the set of non-negative finite measures on $\X$.  Let $\M_a(\X)\subset \M_F(\X)$ be the set of atomic measures on $\X$:
\[
\M_a(\X)=\left\{\sum\limits_{i=1}^n\de_{x_i}: x_1,\cdots, x_n\in\X, n\in\N\right\}.
\]
Then the population process can be represented as:
\[
\nu_t=\sum\limits_{i=1}^{I(t)}\delta_{X_i(t)}.
\]
Let $B(\X)$ denote the totality of bounded and measurable functions on $\X$. Let $B(\M_F(\X))$ (and $B(\M_a(\X))$) be totality of bounded and measurable functions on
$\M_F(\X)$ (and $\M_a(\X)$). For $\nu\in\mathcal{M}_F(\X)$ and $\phi\in B(\X)$, denote by $\langle\nu,\phi\rangle=\int\phi d\nu$.

Let's specify the population processes $(\nu_t^n)_{t>0}$ by introducing a sequence of demographic parameters, for n$\in\N$:
\begin{itemize}
  \item $b_n(x)$ is the rate of birth from an individual with trait $x$.
  \item $d_n(x)$ is the rate of death of an individual with trait $x$ because of ``aging''.
  \item $\al_n(x,y)$ is the competition kernel felt by some individual with trait $x$ from another individual with trait $y$.
  \item $D_n(x,dy)$ is the children's dispersion law from its mother with trait $x$. In particular, it can be decomposed into two parts-local birth at location $x$ and a small portion of migration based on birth, i.e.
  \beq\label{dispersal kernel_Cha2}
  D_n(x,dy)=(1-\e)1_{x=y}+\e m_n(x,dy)1_{x\neq y}.
  \eeq
  Here, $m_n(x,dy)$ is the transition law for migration, which satisfies
  \[
  \int_{y\in\X}m_n(x,dy)=1.
  \]
  We will omit the superscript $\epsilon$ in $D_n$ in the sequel when this leads no ambiguity.
\end{itemize}

Fournier and M\'el\'eard \cite{FM04} formulated a pathwise construction of the BPDL process $\{(\nu_t^n)_{t\geq0}; n\in\N\}$ in terms of Poisson random measures and justified its infinitesimal generator defined for any $\Phi\in B(\M_a(\X))$:
\beq\label{generator BPDL_Cha2}
\bea
L_0^n\Phi(\nu)=&\int_{\X}\nu(dx)\int_{\R^d}\left[\Phi(\nu+\de_{y})-\Phi(\nu)\right]b_n(x)D_n(x,dy)\\
             &+\int_{\X}\nu(dx)\left[\Phi(\nu-\de_x)-\Phi(\nu)\right]\left[d_n(x)+\int_{\X}\al_n(x,y)\nu(dy)\right].
\eea
\eeq
The first term is used to model birth events, while the second term which is nonlinear is interpreted as natural death and competing death.

Instead of studying the original BPDL processes defined by \eqref{generator BPDL_Cha2}, our goal is to study the rescaled processes
 \beq\label{BPDL process scaled_Cha2}
X_t^n:=\frac{\nu_t^n}{n}, \qquad t\geq0
 \eeq
since it provides us a macroscopic approximation when we take the large population limits (we will see later, the initial population is proportional to $n$ in some sense).
The infinitesimal generator of the rescaled BPDL process has the following form, for any $\Phi\in B(\M_F(\X))$:
\beq\label{generator BPDL rescaled_Cha2}
\bea
L^n\Phi(\nu)=&\int_{\X}n\nu(dx)\int_{\R^d}\left[\Phi(\nu+\frac{\de_{y}}{n})-\Phi(\nu)\right]b_n(x)D_n(x,dy)\\
             &+\int_{\X}n\nu(dx)\left[\Phi(\nu-\frac{\de_x}{n})-\Phi(\nu)\right]\left[d_n(x)+\int_{\X}\al_n(x,y)n\nu(dy)\right].
\eea
\eeq

\subsection{Preliminary results}
Let's denote by (A) the following assumptions:
\begin{itemize}
\item[(A1)] There exist $b(x),\, d(x), \,\bar m(x)\in B(\X),\, \al(x,y)\in B(\X\times\X)$ with $\bar m(x)$ a probability density for $x, y\in\X, \,n\in\N$, such that
 \[\bea & 0<b_n(x)\equiv b(x),\qquad 0<d_n(x)\equiv d(x), \qquad m_n(x,y)\leq\bar{m}(y),\\
        &0\leq\al_n(x,y)=\frac{\al(x,y)}{n}
   \eea
 \]

\item[(A2)] $ b(x)-d(x)>0$.
 \end{itemize}

 The first assumption implies that there exist constants $\bar{b}, \,\bar{d}, \,\bar{\al}$ such that $b(x)\leq\bar{b}, \,d(x)\leq\bar{d}, \,\al(x,y)\leq\bar{\al}$. Furthermore, it guarantees the existence of the BPDL process (see \cite{FM04}).

 By neglecting the high order moment, Bolker and Pacala \cite{BP97} use the ``moment closure'' procedure to approximate the stochastic population processes. As we can see from the generator formula \eqref{generator BPDL rescaled_Cha2}, due to the quadratic nonlinear term, it should be enough to set the third order moments to be uniformly bounded and ``close'' the equation up to second order moment . Then Fournier and M\'el\'eard \cite{FM04} obtain a deterministic measure-valued process in the large population limit.

\begin{teor}[Fournier and M\'el\'eard \cite{FM04}, convergence to an integro-differential equation]\label{Theorem LLN_Cha2}
Under the assumption (A1), consider a sequence of processes $(X_t^n)_{t\geq 0}$ defined in \eqref{BPDL process scaled_Cha2}. Suppose that $(X_0^n)$ converges in law to some deterministic finite measure $X_0\in\M_F(\X)$ as $n\to\infty$ and satisfies $\sup\limits_{n\geq1}\E\langle X_0^n,1\rangle^3<\infty$.
Then the sequence of processes $(X_t^n)_{t\geq 0}$ converges in law as $n\to\infty$, on $D([0,\infty),\M_F(\X))$, to a deterministic measure-valued process $(X_t)_{t\geq0}\in C([0,\infty),\M_F(\X))$, where $(X_t)_{t\geq0}$ is the unique solution satisfying
\beq\label{LLN limit_initial_Cha2}
\sup\limits_{t\in[0,T]}\langle X_t,1\rangle<\infty,
\eeq
and for any $\phi\in B(\X)$,
\beq\label{LLN limit_Cha2}
\bea
\langle X_t,\phi\rangle =&\langle X_0,\phi\rangle+\int_0^tds\int_{\X}X_s(dx)b(x)\int_{\R^d}\phi(y)D(x,dy)\\
\qquad &-\int_0^tds\int_{\X}X_s(dx)\phi(x)\left[d(x)+\int_{\X}\al(x,y)X_s(dy)\right].
\eea
\eeq
 \end{teor}


\section{TST on a finite trait space: without mutation}\label{section 2.3}
The trait substitution sequence (TSS) model is a powerful tool in understanding various evolutionary phenomena, such as evolutionary branching which may lead to speciation (see Champagnat and M\'el\'eard  \cite{CM10}). 
Moreover, the population follows the ``hill climbing'' process on the increasing fitness landscape, and holds monomorphic trait on a long time scale. This model is proposed by Metz et al. \cite{MGMJ96} (so called ``invasion implies fixation'') and mathematically studied by Champagnat et al. \cite{Cha06, CL07, MT09}.

Notice that the dispersal kernel $D(x,dz)$ implicitly depends on a parameter $\epsilon$ (see \eqref{dispersal kernel_Cha2}).
Rather than taking large population and rare migration limits simultaneously as in \cite{Cha06}, we justify a so-called trait substitution tree (TST) from a macroscopic point of view. More precisely, we first consider the large population limit to attain a macroscopic approximation of the individual-based model (see Theorem \ref{Theorem LLN_Cha2}). Then, we consider the slow migration limit by a rescaling procedure based on the macroscopic limit. In contrast to the model in Champagnat \cite{Cha06}, the migration rate here is not constrained in terms of the demographic parameter (population size).

Here, the so-called TST process arises under the slow migration limit when we assume the nearest-neighbor competition. Note that a variety of short-term evolution paths can be attained by specifying different competition strengths. In other words, the order of invasion and recovery has no special significance even though in this section we restrict the picture by forward invasion into the fitter direction and backward recovery into the unfit direction along the fitness landscape. However, these paths are indistinguishable on a longer scale-the mutation time scale followed by the next section. Nevertheless, apart from the interesting tree structure the TST model also brings us some insights into speciation phenomena - evolution from a monomorphic ancestor to diverse species.

Denote by (C) the following assumptions:
\begin{itemize}

\item[(C1)] Assume $\X=\{x_0,x_1,x_2,\cdots,x_{L}\}$ comprised of distinct traits with index up to $L\in\N$. Monomorphic initial trait: $X_0^n=\frac{N_0^n}{n}\de_{x_0}$, and $\frac{N_0^n}{n}\stackrel{\text{law}}{\to}\bar\xi(x_0)$ as $n\to\infty$.  

\item[(C2)] Nearest-neighbor competition and migration: $\al(x_i,x_j)=m(x_i,x_j)\equiv 0$ for $|i-j|>1$, and
\beq
x_0\prec x_1\prec\cdots\prec x_{L-1}\prec x_L
\eeq
where $x_{i-1}\prec x_i$ means $f_{i,i-1}>0, f_{i-1,i}<0$ for any $1\leq i\leq L$ with fitness function $f_{i,j}:=b(x_i)-d(x_i)-\al(x_i,x_j)\bar{\xi}(x_j)$, and $\bar{\xi}(x_j):=\frac{b(x_j)-d(x_j)}{\al(x_j,x_j)}$.

\item [(C3)] For any $i\geq 2$,
    \beq
    \frac{i}{b(x_i)-d(x_i)}\geq \frac{1}{f_{i,i-1}}+\frac{1}{f_{i-1,i-2}}+\cdots+\frac{1}{f_{1,0}}.
    \eeq

\item[(C4)] For any $i\geq 0$, $\frac{|f_{i,i+1}|}{f_{i+2,i+1}}<1$, and
     \beq
     \frac{|f_{i,i+1}|}{f_{i+2,i+1}(b(x_i)-d(x_i))}-\frac{1}{f_{i+3,i+2}}>\frac{|f_{i+1,i+2}|}{f_{i+3,i+2}(b(x_{i+1})-d(x_{i+1}))}.
     \eeq
\end{itemize}

Notice that (C3-C4) are just technical assumptions for results in this section but not necessary for results in next section. In fact, assumption (C3) guarantees
that the pattern for fixation of fitter traits is in a form of one-by-one replacements until the fittest trait rather than immediate establishments (see proof of Proposition \ref{TSS_3TYPE} and Proposition \ref{TSS_4TYPE}). (C4) implies that the recovery time of trait $x_i$ is later than that of type $x_{i+1}$ (see Lemma \ref{Lem:C4}).

We first consider the macroscopic limit \eqref{LLN limit_Cha2} which involves the parameter $\epsilon>0$, and rewrite it in another form, for any $\phi\in B(\X)$,
\beq
\bea
\langle X_t^{\epsilon},\phi\rangle =&\langle X_0,\phi\rangle+\epsilon\int_0^tds\int_{\X}X_s^{\epsilon}(dx)b(x)\int_{\X}\left[\phi(y)-\phi(x)\right]m(x,dy)\\
\qquad &+\int_0^tds\int_{\X}X_s^{\epsilon}(dx)\phi(x)\left[b(x)-d(x)-\int_{\X}\al(x,y)X_s^{\epsilon}(dy)\right].
\eea
\eeq

Suppose that the process is supported on a finite trait space \[\X=\{x_0,x_1,\cdots,x_{L}\},
\]
and allow only nearest-neighbour  competition and migration. The infinite population size limit then yields  a dynamical system given by 
\beq\label{tss_X sequence}
\bea
\xi_t(x_i)=
&\xi_0(x_i)+\int_0^t\left[b(x_i)-d(x_i)-\sum_{j=i\pm 1,i}\al(x_i,x_j)\xi_s(x_j)\right]\xi_s(x_i)ds\\
&+\epsilon\int_0^t\sum_{j=i\pm 1}\left[b(x_j)\xi_s(x_j)m(x_j,x_i)-b(x_i)\xi_s(x_i)m(x_i,x_j)\right]ds,\qquad 1\leq i\leq L.
\eea
\eeq
Global  existence and uniqueness of the processes follows from Theorem \ref{Theorem LLN_Cha2}. 

 In the following theorem, we derive  a trait substitution tree model based on the above macroscopic approximation by letting $ \epsilon $ tend to  zero while rescaling  time. 
\begin{teor}\label{TST_deterministic}
  Admit assumptions (A) and (C),  consider the deterministic measure-valued processes $(X_t^{\epsilon})_{t\geq0}$ specified by \eqref{tss_X sequence} on the trait space $\X=\{x_0,x_1,x_2,\cdots,x_{L}\}$, for any $L\in\N$.  Then the sequence of rescaled processes $\left(X_{t\cdot \ln\frac{1}{\epsilon}}^{\epsilon}\right)_{t\geq0}$ converges, as $\epsilon\to 0$, to $\left(U_t\right)_{t\geq0}$ which has the following forms depending on the integer $L$ is even or odd.

\begin{itemize}
\item[\textrm{(i)}] When $L=2l$ for some $l\in\N\cup 0$,
\beq\label{tss_U_t_even}
U_t\equiv
\left\{
  \begin{array}{ll}
    \bar{\xi}(x_0)\de_{x_0} & \textrm{for}~ 0\leq t\leq I_1,\\
    \bar{\xi}(x_k)\de_{x_k} & \textrm{for}~ I_k< t\leq I_{k+1}, ~k=1, \cdots, L-1,\\
    \bar{\xi}(x_L)\de_{x_L} & \textrm{for}~ I_L< t\leq I_L+S_{L-2},\\
    \sum\limits_{i=j}^l\bar{\xi}(x_{2i})\de_{x_{2i}} & \textrm{for}~ I_{2j+2}+S_{2j}< t\leq I_{2j}+S_{2j-2}, ~j=l-1, \cdots, 1,\\
    \sum\limits_{i=0}^{l}\bar{\xi}(x_{2i})\de_{x_{2i}} & \textrm{for}~ t> I_2+S_0.
  \end{array}
\right.
\eeq
where $I_k=\sum\limits_{i=1}^k \frac{1}{f_{i,i-1}},$ and $S_k=\frac{|f_{k,k+1}|}{f_{k+2,k+1}(b(x_k)-d(x_k))}$.

\item[\textrm{(ii)}] When $L=2l+1$ for some $l\in\N \cup 0$,
\beq
U_t\equiv
\left\{
  \begin{array}{ll}
    \bar{\xi}(x_0)\de_{x_0} & \textrm{for}~ 0\leq t\leq I_1,\\
    \bar{\xi}(x_k)\de_{x_k} & \textrm{for}~ I_k< t\leq I_{k+1}, ~k=1, \cdots, L-1,\\
    \bar{\xi}(x_L)\de_{x_L} & \textrm{for}~ I_L< t\leq I_L+S_{L-2},\\
    \sum\limits_{i=j}^{l+1}\bar{\xi}(x_{2i-1})\de_{x_{2i-1}} & \textrm{for}~ I_{2j+1}+S_{2j-1}< t\leq I_{2j-1}+S_{2j-3}, ~j=l, \cdots 2, \\
    \sum\limits_{i=1}^{l+1}\bar{\xi}(x_{2i-1})\de_{x_{2i-1}} & \textrm{for}~ t> I_3+S_1.
  \end{array}
\right.
\eeq
\end{itemize}
\end{teor}

 \begin{figure}[hbtp]
 \centering
 \includegraphics[width=400pt]{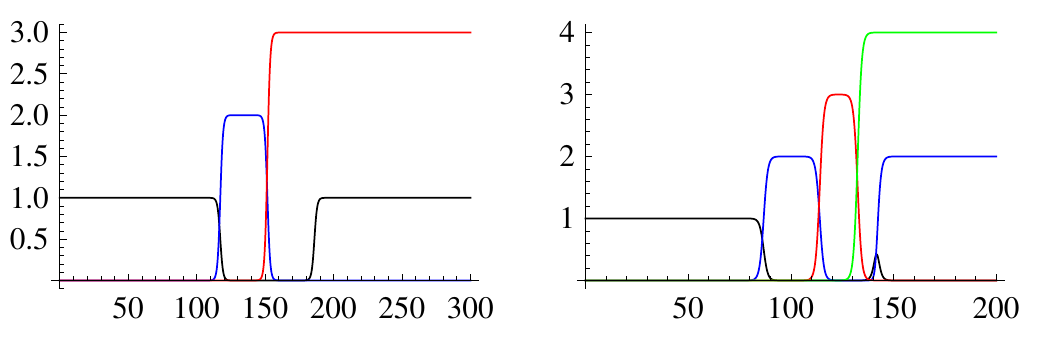}
  \def\svgwidth{400pt}
 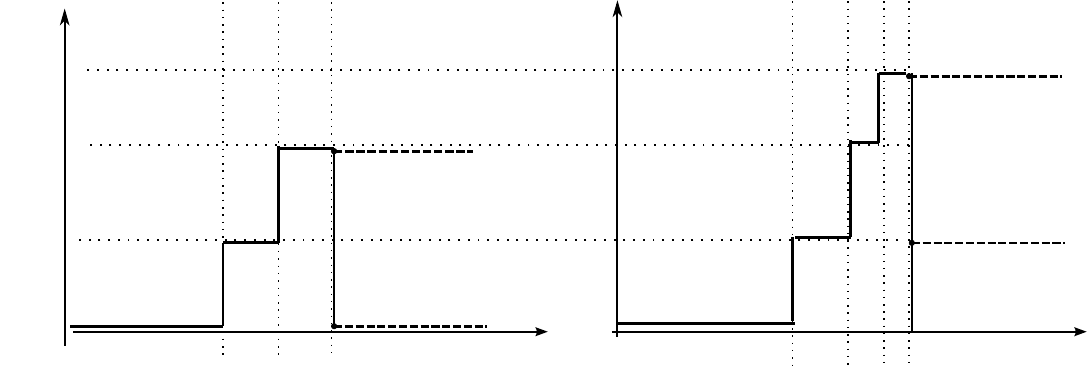
 \caption{\footnotesize Numerical simulations of evolution of a dynamical system with monomorphic initial type and finite trait space (the upper left one has $\X=\{x_0, x_1, x_2\}$ while the upper right one has $\X=\{x_0, x_1, x_2, x_3\}$). Curves describing $\xi_t(x_0), \xi_t(x_1), \xi_t(x_2), \xi_t(x_3)$ are colored black, blue, red, green, resp.. The equilibrium configuration for the first case is $\de_{x_0}+3\de_{x_2}$ and is $2\de_{x_1}+4\de_{x_3}$ for the second one. The lower panel gives their corresponding ``trait substitution tree" structure.}
 \label{treesimulation}
 \end{figure}

\begin{rem}\label{Remark:TST_finite} (1) As time passes on, the limiting process $(U_t)_{t\geq 0}$ starts with monomorphic substitutions up to the domination of the fittest trait. Afterwards, the relatively unfit traits start to recover along the fitness decreasing direction. From the fittest trait back to the initial one every second one appears in the limit. For instance, when $\X=\{x_0,x_1,x_2\}$, the stable configuration has support $\{x_0,x_2\}$; when $\X=\{x_0,x_1,x_2,x_3\}$, the stable configuration has support $\{x_1,x_3\}$ (see Figure \ref{treesimulation}). This is because the competition is restricted between nearest neighbors, and the trait on the right hand side is always fitter than the traits on the left.
\begin{figure}[hbtp]
 \centering
 \def\svgwidth{350pt}
 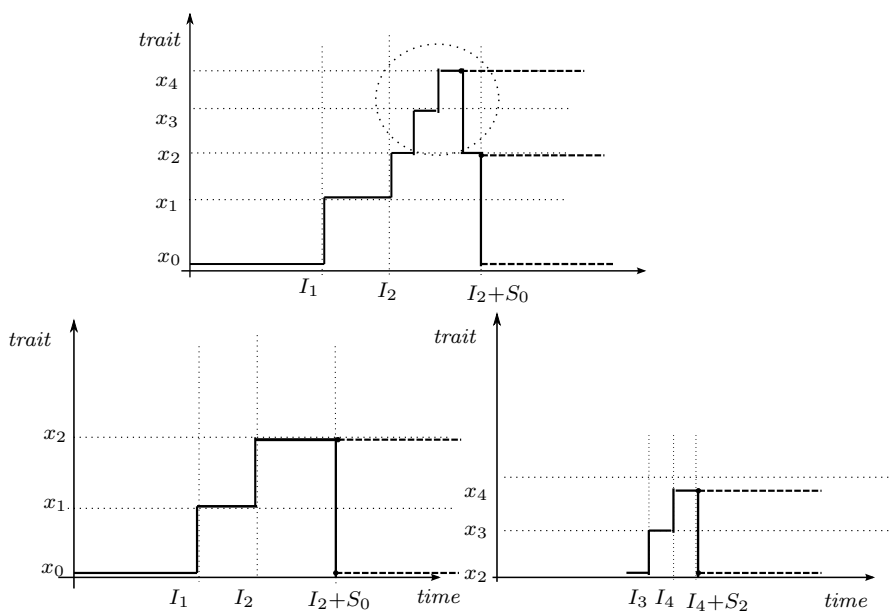
 \caption{\footnotesize Trait substitution tree constructed by embedding excursions.}
 \label{treeconstruction}
 \end{figure}

(2) The TST process indexed by $L+2$ can be constructed from the TST process indexed by $L$ by adding a three-type sub-tree on top of it. For instance, it is shown in the Figure \ref{treeconstruction} that the TST (when $L=4$) can be constructed from a smaller TST (when $L=2$) by connecting another excursion consisting of traits $\{x_2,x_3,x_4\}$.

\end{rem}

We postpone the proof of the above result to Section \ref{section 2.5.1}.


\section{TST on an infinite trait space: with mutation}\label{section 2.4}
In Section \ref{section 2.3} we analyze a continuous-mass population on a finite trait space defined by equation \eqref{tss_X sequence}. On the way towards its equilibrium configuration, under some restrictive conditions, a deterministic evolutionary picture arises on the slow migration time scale $O\left(\ln\frac{1}{\epsilon}\right)$. 

In order to generalize the process to infinite trait space, we introduce another evolutionary mechanism, 
mutation of a trait $x$ with a transition kernel $p(x,dh)$ for mutant variation such that $x+h\in\X$. 
Notice that the essential difference between mutation and migration is that mutation creates some new trait,
 while migration is only allowed among the existing traits. More precisely, we specify a new model $\{X^{\epsilon,\sigma}_t, t\geq 0\}$ on $D([0,\infty),\M_F(\X))$ with the following infinitesimal generator, for any $\epsilon, \sigma\geq0$ and proper test functions $F$ and $\phi$,

\beq\label{generator_X_epsilon,sigma}
\bea
L^{\epsilon,\sigma}F(\nu)=
&\int_{\X}\left[b(x)-d(x)-\int_{\X}\al(x,y)\nu(dy)\right]\frac{\delta F(\nu;x)}{\delta\nu}\nu(dx)\\
&+\epsilon\int_{\X} A\Big(\frac{\delta F(\nu;x)}{\delta\nu}\Big)\nu(dx)\\
&+\sigma\int_{\X}\int_{\R^d}\left[F(\nu+\rho\delta_{x+h})-F(\nu)\right]\mu(x)p(x,dh)\nu(dx),
\eea
\eeq
where the derivative  of $F$ is defined by
\beq
\frac{\delta F(\nu;x)}{\delta\nu}=\lim_{\varepsilon\to 0+}\frac{F(\nu+\varepsilon\delta_x)-F(\nu)}{\varepsilon}
\eeq
and the operator $A$ coincides with the migration term in \eqref{tss_X sequence}
\beq
A\phi(x)=\int_{\X}\big[\phi(y)-\phi(x)\big]1_{\{y\in\textrm{supp}\{\nu\}\}}m(x,dy).
\eeq
The first term of the generator describes the local regulation of population dynamics. The second term describes migration among supporting trait sites. Note that  migration is not restricted to birth events 
any more as in \eqref{tss_X sequence}, which is reasonable if we interpret them as a changes in gene expression. 
The last term creates a new mutant trait to the current population. The mutant mass is specified by a 
magnitude of $\rho>0$, which can be taken to zero in a final step. The non-negative function
 $\mu(x)$ describes the mutation rate of the resident trait $x$. The parameters $\epsilon$ and $\sigma$ are used to rescale the strength of migration and mutation of the population. For any fixed $\epsilon,\, \sigma>0$, the process $\{X^{\epsilon,\sigma}_t, t\geq 0\}$ can be obtained as a large population
limit (as $n\to\infty$) of the processes specified by the  generator
\beq
\bea
 L^{n,\epsilon,\sigma} F(\nu)
    &= \int_{\X}\left[F(\nu+\frac{\delta_x}{n})-F(\nu)\right]b(x)n\nu(dx)\\
    &+\int_{\X}\left[F(\nu-\frac{\delta_x}{n})-F(\nu)\right]\left(d(x)+\int_{\X}\alpha(x,y)\nu(dy)\right){n}\nu(dx)\\
    &+\epsilon\int_{\X}\int_{\X}\left[F(\nu+\frac{\de_{y}}{n}-\frac{\de_{x}}{n})-F(\nu)\right]1_{\{y\in
     \textrm{supp}\{\nu\}\}}m(x,dy)n\nu(dx)\\
    &+\sigma\int_{\X}\int_{\R^d}\left[F(\nu+\rho\de_{x+h})-F(\nu)\right]\mu(x)p(x,dh)\nu(dx).
\eea
\eeq

 For more discussion on discontinuous superprocesses with a general branching mechanism, one can refer to \cite{Li10}. We will not expand the discussion here.

The following assumptions (D) ensure that the limiting TST process is well-defined.
\begin{itemize}

\item[(D1)] For any given set of distinct traits $\{x_0, x_1,\cdots, x_n\}\subset\X, n\in\N$, there exists a total order permutation
\beq
x_{n_0}\prec x_{n_1}\prec\cdots\prec x_{n_{n-1}}\prec x_{n_n},
\eeq
where $x\prec y$ means that the fitness functions satisfy $f(x,y):=b(x)-d(x)-\al(x,y)\bar\xi(y)<0$, and $f(y,x):=b(y)-d(y)-\al(y,x)\bar\xi(x)>0$.

For simplicity, we always assume $x_0^{(n)}\prec x_1^{(n)}\prec\cdots\prec x_n^{(n)}$ with $x^{(n)}_i=x_{n_i}$, $1\leq i\leq n$. Every time  a new trait $x$ appears whose fitness is between  $x_j^{(n)}$ and $x_{j+1}^{(n)}$ for some $0\leq j\leq n$, we relabel the traits as follows
\beq
x^{(n+1)}_0\prec x^{(n+1)}_1\prec\cdots\prec x^{(n+1)}_n\prec x^{(n+1)}_{n+1},
\eeq
where $x^{(n+1)}_i=x^{(n)}_i$ for $0\leq i\leq j$, $x^{(n+1)}_{j+1}=x$ and $x^{(n+1)}_{i}=x^{(n)}_{i-1}$ for $j+2\leq i\leq n+1$.

\item[(D2)] Competition and migration only occurs between nearest neighbors, i.e., for totally ordered traits in (D1), we have $m(x_i^{(n)},x_j^{(n)})=\al(x_i^{(n)},x_j^{(n)})\equiv 0$ for $\mid i-j\mid>1$.
\end{itemize}

Notice that assumptions (C3-C4) 
provide a convenient setting for which the  evolutionary path on the migration time scale can be easily 
identified. More complex situations can, however, be analyzed in a similar way and lead to qualitatively 
similar results.


We now give a description of the limiting process on the mutation time-scale.
\begin{Def}\label{TST_Definition_infinite tree} A $\M_F(\X)$-valued Markov jump process $\{\Gamma_t: t\geq 0\}$ characterized as follows is called a trait substitution tree with the ancestor $\Gamma_0=\bar\xi({x_0})\delta_{x_0}$.
\begin{itemize}
\item[(i)] For any non-negative integer $l$, it jumps
    from $\Gamma^{(2l)}:=\sum_{i=0}^l\bar\xi(x^{(2l)}_{2i})\delta_{x^{(2l)}_{2i}}$
    to $\Gamma^{(2l+1)}$\\
    with transition rate $\bar\xi(x^{(2l)}_{2k})\mu(x^{(2l)}_{2k})p(x^{(2l)}_{2k},dh)$ for any $0\leq k\leq l$, where
    \begin{itemize}
    \item $\Gamma^{(2l+1)}=\sum_{i=1}^j\bar\xi(x^{(2l)}_{2i-1})\delta_{x^{(2l)}_{2i-1}}+\bar\xi(x^{(2l)}_{2k}+h)\delta_{x^{(2l)}_{2k}+h}+\sum_{i=j+1}^l\bar\xi(x^{(2l)}_{2i})\delta_{x^{(2l)}_{2i}}$\\
        \vspace{4mm}
        if there exists $0\leq j\leq l$ s.t.
        $x^{(2l)}_{2j}\prec x^{(2l)}_{2k}+h\prec x^{(2l)}_{2j+1}$,\\
    \item $\Gamma^{(2l+1)}=\sum_{i=1}^j\bar\xi(x^{(2l)}_{2i-1})\delta_{x^{(2l)}_{2i-1}}+\sum_{i=j}^l\bar\xi(x^{(2l)}_{2i})\delta_{x^{(2l)}_{2i}}$\\ \vspace{4mm}
        if there exists $0\leq j\leq l$ s.t.
        $x^{(2l)}_{2j-1}\prec x^{(2l)}_{2k}+h\prec x^{(2l)}_{2j}$.
    \end{itemize}
    Then, we relabel the trait sequence according to the total order relation as in (D1):
    \beq
    x_0^{(2l+1)}\prec x_1^{(2l+1)}\prec\cdots\prec  x_{2l}^{(2l+1)}\prec x_{2l+1}^{(2l+1)},
    \eeq
    where in associate with the first case
    \begin{align*}
    &x_i^{(2l+1)}:=x_i^{(2l)} ~\textrm{for} ~0\leq i\leq 2j,\qquad x^{(2l+1)}_{2j+1}:=x^{(2l)}_{2k}+h,\\
    &x_i^{(2l+1)}:= x_{i-1}^{(2l)}~\textrm{for} ~2j+2\leq i\leq 2l+1,
    \end{align*}
    and in associate with the second case
    \begin{align*}
    &x_i^{(2l+1)}:=x_i^{(2l)} ~\textrm{for} ~0\leq i\leq 2j-1,\quad x^{(2l+1)}_{2j}:=x^{(2l)}_{2k}+h,\\
    &x_i^{(2l+1)}:= x_{i-1}^{(2l)} ~\textrm{for} ~2j+1\leq i\leq 2l+1.
    \end{align*}

\item[(ii)] For non-negative integer $l$, it jumps
    from $\Gamma^{(2l+1)}:=\sum_{i=1}^{l+1}\bar\xi(x^{(2l+1)}_{2i-1})\delta_{x^{(2l+1)}_{2i-1}}$ to $\Gamma^{(2l+2)}$\\
    \vspace{4mm}
    with transition rate $\bar\xi(x^{(2l+1)}_{2k-1})\mu(x^{(2l+1)}_{2k-1})p(x^{(2l+1)}_{2k-1},dh)$ for any $1\leq k\leq l+1$, where
    \vspace{4mm}
    \begin{itemize}
    \item $\Gamma^{(2l+2)}=\sum_{i=1}^j\bar\xi(x^{(2l+1)}_{2(i-1)})\delta_{x^{(2l+1)}_{2(i-1)}}+\bar\xi(x^{(2l+1)}_{2k-1}+h)\delta_{x^{(2l+1)}_{2k-1}+h}+\sum_{i=j+1}^{l+1}\bar\xi(x^{(2l+1)}_{2i-1})\delta_{x^{(2l+1)}_{2i-1}}$\\
        \vspace{4mm}
        if there exists $1\leq j\leq l+1$ s.t.
        $x^{(2l+1)}_{2j-1}\prec x^{(2l+1)}_{2k-1}+h\prec x^{(2l+1)}_{2j}$,\\
    \item  $\Gamma^{(2l+1)}=\sum_{i=1}^j\bar\xi(x^{(2l+1)}_{2(i-1)})\delta_{x^{(2l+1)}_{2(i-1)}}+\sum_{i=j}^{l+1}\bar\xi(x^{(2l+1)}_{2i-1})\delta_{x^{(2l+1)}_{2i-1}}$\\ \vspace{4mm}
        if there exists $1\leq j\leq l+1$ s.t.
        $x^{(2l+1)}_{2j-2}\prec x^{(2l+1)}_{2k-1}+h\prec x^{(2l+1)}_{2j-1}$.
    \end{itemize}
    Then, we relabel the trait sequence according to the total order relation as in (D1):
    \beq
    x_0^{(2l+2)}\prec x_1^{(2l+2)}\prec\cdots\prec  x_{2l+1}^{(2l+2)}\prec x_{2l+2}^{(2l+2)},
    \eeq
    where in the first case
    \begin{align*}
    &x_i^{(2l+2)}:=x_i^{(2l+1)} ~\textrm{for} ~0\leq i\leq 2j-1,\quad x^{(2l+2)}_{2j}:=x^{(2l+1)}_{2k-1}+h,\\
    &x_i^{(2l+2)}:= x_{i-1}^{(2l+1)}~\textrm{for} ~2j+1\leq i\leq 2l+2,
    \end{align*}
    and in  the second case
    \begin{align*}
    &x_i^{(2l+2)}:=x_i^{(2l+1)} ~\textrm{for} ~0\leq i\leq 2j-2,\quad x^{(2l+2)}_{2j-1}:=x^{(2l+1)}_{2k-1}+h,\\
    &x_i^{(2l+2)}:= x_{i-1}^{(2l+1)}~\textrm{for} ~2j\leq i\leq 2l+2.
    \end{align*}
\end{itemize}

\end{Def}

\begin{figure}[hbtp]
 \centering
 \def\svgwidth{450pt}
 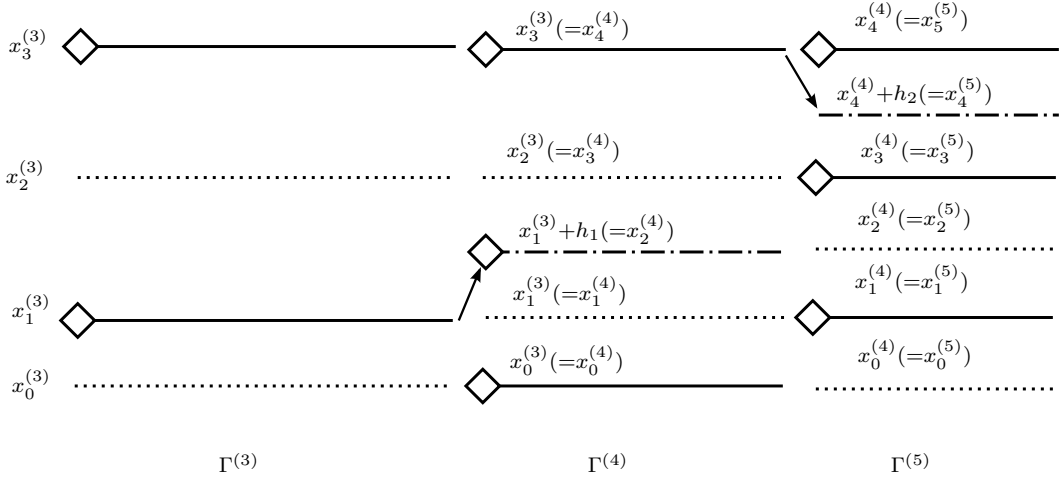
 \caption{\footnotesize A partial path ($\Gamma^{(3)}\to\Gamma^{(4)}\to\Gamma^{(5)}$) of the supporting set of the Trait Substitution Tree process defined by Definition \ref{TST_Definition_infinite tree}.
  Solid lines with ``diamond'' get fixed in the TST process, dotted lines denote virtual traits which are temporarily lost, while dashed lines denote new mutant traits.  Arrows denote creation of new mutant types. Starting with state $\Gamma^{(3)}$, the first mutant arises from trait $x^{(3)}_1$, and new trait fits into $x^{(3)}_1\prec x^{(3)}_1+h_1\prec x^{(3)}_2$. According to Definition \ref{TST_Definition_infinite tree} (ii), traits $x^{(3)}_0,\,x^{(3)}_1+h_1, \,x^{(3)}_3$ get fixed. 
 Consecutively, based on state configuration $\Gamma^{(4)}$, the second mutant arises from trait $x^{(4)}_4$ , and new trait fits into $x^{(4)}_3\prec x^{(4)}_4+h_2\prec x^{(4)}_4$. 
 By Definition \ref{TST_Definition_infinite tree} (i), traits $x^{(4)}_1,\,x^{(4)}_3, \,x^{(4)}_4$ get fixed.}  
 \label{TSTpath}
 \end{figure}

\begin{rem} (see Figure \ref{TSTpath}).
 According to the definition, the new configuration is constructed in a way that every second trait gets
stabilized when one ``looks down''  from the fittest trait along the fitness landscape. Once a mutant is 
inserted between two trait levels (say, $i$ and $i+1$), we relabel all the traits above 
 the mutant's level. However, the mutation only alters the configuration below the $i+1$th level and  
not above. To some extent, the construction here is similar to the look-down idea of Coalescent processes (see \cite{DK99}).  
\end{rem}

\begin{teor}\label{TST_theorem_infinite trait}
Admit assumption (A) and (D), and consider processes $\{X^{\epsilon,\sigma}_t, t\geq 0\}$ described by the generator \eqref{generator_X_epsilon,sigma}. Suppose that $X_0^{\epsilon,\sigma}=\xi^{\epsilon}(x_0)\delta_{x_0}$ and $\xi^{\epsilon}(x_0)\rightarrow\bar\xi(x_0)$ in law and $\rho\to0$, as $\epsilon\to 0$. If it holds that
\beq
\frac{1}{\sigma}\gg\ln\frac{1}{\epsilon},
\eeq
then $(X^{\epsilon,\sigma}_{\frac{t}{\sigma}})_{t\geq 0}$ converges, as $\epsilon\to 0$, 
to the trait substitution tree $(\Gamma_{t})_{t\geq 0}$  given in Definition 
\ref{TST_Definition_infinite tree}. Convergence is  in the sense of finite dimensional distribution.
\end{teor}

We postpone the proof of the above result in Section \ref{section 2.5.2}.


\section{Outline of proofs}\label{section 2.5}
\subsection{Proof of Theorem \ref{TST_deterministic}}\label{section 2.5.1}

In this section we present the proofs of the results of Section $\ref{section 2.3}$. 
The main idea behind the proofs is that the migration spreads linearly and the nearest neighbor competitive growth spreads exponentially fast.
Before proving Theorem \ref{TST_deterministic} we state some preliminary results which are key ingredients for the proof of Theorem \ref{TST_deterministic}.

The following lemma ensures  the non-coexistence condition for a dimorphic Lotka-Volterra system. 
We  give  the proof in the appendix.
\begin{lem}\label{lemma_noncoexistence condition}
Consider a dimorphic system
\begin{align} \label{only-eq}
\left\{
\begin{array}{l}
\dot{\xi}_t(x_i)=\Big(b(x_i)-d(x_i)-\alpha(x_i,x_i)\xi_t(x_i)-\alpha(x_i,x_{i+1})\xi_t(x_{i+1})\Big)\xi_t(x_i)\\
\dot{\xi}_t(x_{i+1})=\Big(b(x_{i+1})-d(x_{i+1})-\alpha(x_{i+1},x_i)\xi_t(x_i)-\alpha(x_{i+1},x_{i+1})\xi_t(x_{i+1})\Big)\xi_t(x_{i+1}),
\end{array}
\right.
\end{align}
with some positive initial condition.
 If $f_{i,i+1}<0,\,f_{i+1,i}>0$, then $\big(0, \bar\xi(x_{i+1})\big)$ is the only stable equilibrium of 
\eqref{only-eq}.
\end{lem}

The following two propositions are used to prove Theorem \ref{TST_deterministic}.
\begin{prop}\label{TSS_3TYPE} Under the assumptions of   Theorem \ref{TST_deterministic}, for 
the case when $L=2$ (i.e. $\X=\{x_0,x_1,x_2\}$), the limit process $(U_t)_{t\geq 0}$ has the form
\beq\label{tss_3type}
U_t\equiv
\left\{
  \begin{array}{ll}
    \bar{\xi}(x_0)\de_{x_0} & \textrm{for}~ 0\leq t\leq I_1, \\
    \bar{\xi}(x_1)\de_{x_1} & \textrm{for}~ I_1< t\leq I_2, \\
    \bar{\xi}(x_2)\de_{x_2} & \textrm{for}~ I_2< t\leq I_2+S_0,\\
     \bar{\xi}(x_0)\de_{x_0}+\bar{\xi}(x_2)\de_{x_2} & \textrm{for}~ t> I_2+S_0,
  \end{array}
\right.
\eeq
where $I_1=\tfrac{1}{f_{1,0}}$, $I_2=\tfrac{1}{f_{1,0}}+\tfrac{1}{f_{2,1}}$, and 
$S_0=\tfrac{|f_{0,1}|}{f_{2,1}(b(x_0)-d(x_0))}$.
\end{prop}

\begin{proof}
\textbf{(a)} First, suppose that the population consists  of only two types,  $\X=\{x_0, x_1\}$. 
We divide the entire invasion period into four steps, as shown in Figure \ref{dynamics_2type}.

 \begin{figure}[hbtp]
 \centering
 \def\svgwidth{250pt}
 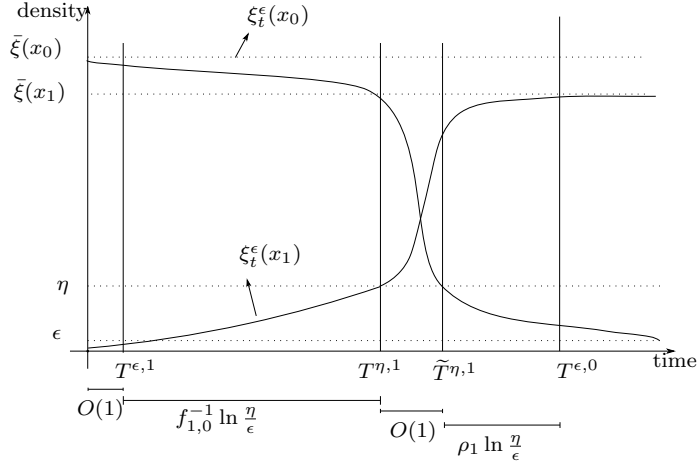
 \caption{{\footnotesize Four-step invasion analysis for a dimorphic system }}.
 \label{dynamics_2type}
 \end{figure}

Let $\xi_t^{\epsilon}(x_0):=\langle X_t^{\epsilon}, 1_{\{x_0\}}\rangle$ and $\xi_t^{\epsilon}(x_1):=\langle X_t^{\epsilon}, 1_{\{x_1\}}\rangle$. From \eqref{tss_X sequence} one obtains
\beq\label{x_0 equation}
\bea
\dot {\xi}_t^{\epsilon}(x_0)=&\big(b(x_0)-d(x_0)-\al(x_0,x_0)\xi_t^{\epsilon}(x_0)-\al(x_0,x_1)\xi_t^{\epsilon}(x_1)\big)\xi_t^{\epsilon}(x_0)\\
&-\epsilon\xi_t^{\epsilon}(x_0)b(x_0)m(x_0,x_1)+\epsilon\xi_t^{\epsilon}(x_1)b(x_1)m(x_1,x_0),
\eea
\eeq
and
\beq\label{x_1 equation}
\bea
\dot {\xi}_t^{\epsilon}(x_1)=&\big(b(x_1)-d(x_1)-\al(x_1,x_0)\xi_t^{\epsilon}(x_0)-\al(x_1,x_1)\xi_t^{\epsilon}(x_1)\big)\xi_t^{\epsilon}(x_1)\\
&-\epsilon\xi_t^{\epsilon}(x_1)b(x_1)m(x_1,x_0)+\epsilon\xi_t^{\epsilon}(x_0)b(x_0)m(x_0,x_1),
\eea
\eeq
where $\xi_0^{\epsilon}(x_0)=\bar\xi(x_0)$ and $\xi_0^{\epsilon}(x_1)=0$.

 \noindent\textbf{Step 1.} For any fixed $\eta>0$, $\forall \, 0<\epsilon<\eta$, let $T^{\epsilon,1}$ be the time when  $\big(\xi_t^{\epsilon}(x_0),\xi_t^{\epsilon}(x_1)\big)$ leaves the $\epsilon$-neighborhood of $(\bar{\xi}(x_0), 0)$, i.e.  
$$
T^{\epsilon,1}=\inf\big\{t\geq0: \xi_t^{\epsilon}(x_1)>\epsilon, \, \textrm{or}~ \xi_t^{\epsilon}(x_0)<\bar{\xi}(x_0)-\epsilon\big\}.
$$
 From $\eqref{x_1 equation}$ it follows that , for $t<T^{\epsilon,1}$, $\xi_t^{\epsilon}(x_1)$ satisfies the 
following differential inequality:
\beq
\bea
\dot{\xi}_t^{\epsilon}(x_1)\geq&\big(b(x_1)-d(x_1)-\al(x_1,x_0)\bar{\xi}(x_0)-\epsilon\al(x_1,x_1)-\epsilon b(x_1)m(x_1,x_0)\big)\xi_t^{\epsilon}(x_1)\\
&+\epsilon\big(\bar{\xi}(x_0)-\epsilon)b(x_0)m(x_0,x_1)\\
=&\big(f_{1,0}-\epsilon(\al(x_1,x_1)+ b(x_1)m(x_1,x_0))\big)\xi_t^{\epsilon}(x_1)+\epsilon\big(\bar{\xi}(x_0)-\epsilon)b(x_0)m(x_0,x_1).
\eea
\eeq
Since $f_{1,0}=b(x_1)-d(x_1)-\al(x_1,x_0)\bar\xi(x_0)>0$, we can choose $\epsilon$ sufficiently small 
so that the first term on the right hand side of the above inequality is positive. 
Omitting this positive term, 
one sees  that $\check{\xi}_t(x_1)\leq \xi_t^{\epsilon}(x_1)$, where $\check{\xi}_0(x_1)=0$, and
\beq
\dot{\check{\xi}}_t(x_1)=\epsilon \big(\bar{\xi}(x_0)-\epsilon)b(x_0)m(x_0,x_1).
\eeq
Thus, $T^{\epsilon,1}$ can be bounded from above by $\check{T}^{\epsilon,1}=\big((\bar\xi(x_0)-\epsilon)b(x_0)m(x_0,x_1)\big)^{-1}$, which is the time when $\check {\xi}_t(x_1)$ reaches the level  
$\epsilon$-level. Thus,  $T^{\epsilon,1}$ is of order $O(1)$.

\textbf{Step 2.} After  time $T^{\epsilon,1}$, we consider the evolution of the population $\big(\xi_t^{\epsilon}(x_0), \xi_t^{\epsilon}(x_1)\big)$ until the time (denoted  by $T^{\eta,1}$) when it leaves the 
$\eta$-neighborhood of $(\bar{\xi}(x_0), 0)$.
From \eqref{x_1 equation},  omitting the term $\epsilon\xi_t^{\epsilon}(x_0)b(x_0)m(x_0,x_1)$, we get
\beq\label{Invastion_step2_lower bound}
\bea
&\dot{\xi}_t^{\epsilon}(x_1)\\
&\geq \big(b(x_1)-d(x_1)-\al(x_1,x_0)\xi_t^{\epsilon}(x_0)-\al(x_1,x_1)\xi_t^{\epsilon}(x_1)\big)\xi_t^{\epsilon}(x_1)-\epsilon\xi_t^{\epsilon}(x_1)b(x_1)m(x_1,x_0)\\
&\geq\big(b(x_1)-d(x_1)-\al(x_1,x_0)\bar{\xi}(x_0)-\eta\al(x_1,x_1)\big)\xi_t^{\epsilon}(x_1)-\eta\xi_t^{\epsilon}(x_1)b(x_1)m(x_1,x_0)\\
&=\big(f_{1,0}-\eta\check C\big)\xi_t^{\epsilon}(x_1),
\eea
\eeq
where $\check C=\al(x_1,x_1)+b(x_1)m(x_1,x_0)$.
On the other hand, by omitting some negative terms in \eqref{x_1 equation}, we get
\beq\label{Invastion_step2_upper bound}
\bea
\dot{\xi}_t^{\epsilon}(x_1)&\leq \big(b(x_1)-d(x_1)-\al(x_1,x_0)\xi_t^{\epsilon}(x_0)\big)\xi_t^{\epsilon}(x_1)+\epsilon\xi_t^{\epsilon}(x_0)b(x_0)m(x_0,x_1)\\
&\leq \big(b(x_1)-d(x_1)-\al(x_1,x_0)(\bar{\xi}(x_0)-\eta)\big)\xi_t^{\epsilon}(x_1)+\epsilon\bar{\xi}(x_0)b(x_0)m(x_0,x_1)\\
&\leq (f_{1,0}+\eta\hat C)\xi_t^{\epsilon}(x_1),
\eea
\eeq
where $\hat C=\al(x_1,x_0)+\bar{\xi}(x_0)b(x_0)m(x_0,x_1)$.

Applying Gronwall's inequality to \eqref{Invastion_step2_lower bound} and \eqref{Invastion_step2_upper bound}, the density $\xi_t^{\epsilon}(x_1)$, starting with $\xi_{T^{\epsilon,1}}^{\epsilon}
(x_1)=\epsilon$, can be bounded from below by $\check {\xi}_t(x_1)$ and from above by $\hat{\xi}_t(x_1)$, which satisfy the  equations
\beq
\dot{\check {\xi}}_t(x_1)=(f_{1,0}-\check C\eta)\check {\xi}_t(x_1),
\eeq
and
\beq
\dot{\hat {\xi}}_t(x_1)=(f_{1,0}+\hat C\eta)\hat {\xi}_t(x_1),
\eeq
with initial conditions $\check{\xi}_{T^{\epsilon,1}}
(x_1)=\hat{\xi}_{T^ {\epsilon,1}}
(x_1)=\epsilon$, respectively.

The times  
needed for $\check {\xi}_t(x_1)$ and $\hat {\xi}_t(x_1)$ to reach the $\eta$-level can be computed
 explicitly.  They are given by  $\check T^{\eta,1}-T^{\epsilon,1}=(f_{1,0}-\check C\eta)^{-1}\ln\frac{\eta}
{\epsilon}$ and $\hat T^{\eta,1}-T^{\epsilon,1}=(f_{1,0}+\hat C\eta)^{-1}\ln\frac{\eta}{\epsilon}$, 
respectively. 
Since $\hat T^{\eta,1}<T^{\eta,1}<\check T^{\eta,1}$, for any $\eta>0$, $T^{\eta,1}-T^{\epsilon,1}$ is of order $f_{1,0}^{-1}\ln\frac{1}{\epsilon}$.

\textbf{Step 3.} From \eqref{tss_X sequence}, the process $\big(\xi_t^{\epsilon}(x_0), \xi_t^{\epsilon}(x_1)\big)$ starting from time $T^{\eta,1}$
converges, as $\epsilon\to 0$, to the solution of the following system: 
\beq
\begin{cases}
&\dot\xi_t(x_0)=\big(b(x_0)-d(x_0)-\al(x_1,x_0)\xi_t(x_0)-\al(x_1,x_1)\xi_t(x_1)\big)\xi_t(x_0)\\
&\dot\xi_t(x_1)=\big(b(x_1)-d(x_1)-\al(x_1,x_0)\xi_t(x_0)-\al(x_1,x_1)\xi_t(x_1)\big)\xi_t(x_1),
\end{cases}
\eeq
which has a nontrivial initial value $\xi_{T^{\eta,1}}(x_1)=\eta$, and $\xi_{T^{\eta,1}}(x_0)\in(\bar{\xi}(x_0)-\eta, \bar{\xi}(x_0)+\eta)$.
By Lemma \ref{lemma_noncoexistence condition}, this dimorphic system has a unique stable equilibrium $(0,\bar{\xi}(x_1))$ under the assumption that $f_{1,0}>0$, $f_{0,1}<0$.
Let $\widetilde T^{\eta,1}$ be the time when $(\xi_t^{\epsilon}(x_0), \xi_t^{\epsilon}(x_1))$ enters the 
$\eta$-neighborhood of the equilibrium $(0, \bar\xi(x_1))$, i.e. 
$
\xi_{\widetilde T^{\eta,1}}(x_0)=\eta$. Since $\eta$ is a given fixed constant, $\widetilde T^{\eta,1}-T^{\eta,1}$ is of order $O(1)$ as $\epsilon\to 0$.

\textbf{Step 4.} After time $\widetilde T^{\eta,1}$, we consider the time needed 
for $x_1$ to get fixated (i.e. for $x_0$ gets absorbed at  $0$). From \eqref{x_0 equation}, one obtains 
the differential lower bound: 
\beq\label{Invastion_step4_lower bound}
\bea
&\dot{\xi}_t^{\epsilon}(x_0)
\\&\geq \big(b(x_0)-d(x_0)-\al(x_0,x_0)\xi_t^{\epsilon}(x_0)-\al(x_0,x_1)\xi_t^{\epsilon}(x_1)\big)\xi_t^{\epsilon}(x_0)-\epsilon\xi_t^{\epsilon}(x_0)b(x_0)m(x_0,x_1)\\
&\geq\big(b(x_0)-d(x_0)-\eta\al(x_0,x_0)-\al(x_0,x_1)\bar{\xi}(x_1)\big)\xi_t^{\epsilon}(x_0)-\eta b(x_0)m(x_0,x_1)\xi_t^{\epsilon}(x_0)\\
&=\big(f_{0,1}-\eta\check C\big)\xi_t^{\epsilon}(x_0),
\eea
\eeq
where $\check C=\al(x_0,x_0)+b(x_0)m(x_0,x_1)$.
As for the upper bound, we observe that
\beq\label{Invastion_step4_upper bound}
\bea
\dot{\xi}_t^{\epsilon}(x_0)&\leq \big(b(x_0)-d(x_0)-\al(x_0,x_1)\xi_t^{\epsilon}(x_1)\big)\xi_t^{\epsilon}(x_0)+\epsilon\xi_t^{\epsilon}(x_1)b(x_1)m(x_1,x_0)\\
&\leq\big(b(x_0)-d(x_0)-\al(x_0,x_1)(\bar{\xi}(x_1)-\eta)\big)\xi_t^{\epsilon}(x_0)+\epsilon\bar{\xi}(x_1)b(x_1)m(x_1,x_0)\\
&\leq (f_{1,0}+\eta\hat C)\xi_t^{\epsilon}(x_0),
\eea
\eeq
where $\hat C=\al(x_0,x_1)+\bar{\xi}(x_1)b(x_1)m(x_1,x_0)$.

Applying  again Gronwall's inequality to \eqref{Invastion_step4_lower bound} and \eqref{Invastion_step4_upper bound}, we see 
that $\xi_t^{\epsilon}(x_0)$, 
starting with $\xi_{\widetilde T^{\eta,1}}^{\epsilon}(x_0)=\eta$, 
can be bounded from below by $\check{\xi}_t(x_0)$ and from above by 
$\hat{\xi}_t(x_0)$, which satisfy the  equations
\beq\label{Invasion_step4_lower process}
\dot{\check {\xi}}_t(x_0)=(f_{0,1}-\check C\eta)\check {\xi}_t(x_0),
\eeq
and
\beq\label{Invasion_step4_upper process}
\dot{\hat {\xi}}_t(x_0)=(f_{0,1}+\hat C\eta)\hat {\xi}_t(x_0),
\eeq
with $\check{\xi}_{\widetilde T^{\eta,1}}
(x_0)=\hat{\xi}_{\widetilde T^ {\eta,1}}
(x_0)=\eta$.

Since $f_{0,1}=b(x_0)-d(x_0)-\al(x_0,x_1)\bar\xi(x_1)<0$, we can choose $\eta$ small enough so 
that $f_{0,1}+\hat C\eta<0$. Therefore, both $\check{\xi}_t(x_0)$ and $\hat{\xi}_t(x_0)$ decay 
exponentially.
For any $\rho_1>0$, the process $\hat{\xi}_t(x_0)$, in time  of order $\rho_1\ln \frac{\eta}{\epsilon}$, reaches the $\epsilon^{-\rho_1(f_{0,1}+\eta\hat C)}$-neighborhood of $0$, while $\check{\xi}_t(x_0)$ 
reaches   the $\epsilon^{-\rho_1(f_{0,1}-\eta\check C)}$-neighborhood of $0$. 
Let $T^{\epsilon,0}:=\widetilde T^{\eta,1}+\rho_1\ln\frac{\eta}{\epsilon}$. Then, 
\beq
\bea
\lim\limits_{\epsilon\to 0}\hat\xi_{T^{\epsilon,0}}^{\epsilon}(x_0)
 &=\lim\limits_{\epsilon\to 0}\hat\xi_{\widetilde T^{\eta,1}}^{\epsilon}(x_0)
   \exp\left((f_{0,1}+\hat C\eta)(T^{\epsilon,0}-\widetilde T^{\eta,1})\right)\\
 &=\lim\limits_{\epsilon\to 0}\eta\exp\left((f_{0,1}+\hat C\eta)\cdot\rho_1\log\frac{\eta}{\epsilon}\right)\\
 &=\lim\limits_{\epsilon\to 0}\epsilon^{-\rho_1f_{0,1}}\cdot O(\eta)\\
 &=0.
\eea
\eeq
Similarly, we obtain $\lim\limits_{\epsilon\to 0}\check\xi_{T^{\epsilon,0}}^{\epsilon}(x_0)=0$. Therefore, $\lim\limits_{\epsilon\to 0}\xi_{T^{\epsilon,0}}^{\epsilon}(x_0)=0$. Therefore the subpopulation at $x_1$ eventually gets fixated as $\epsilon\to 0$.

Combining the  four steps  above, one concludes that the right time scale for the fitter population 
$x_1$ to get fixated is
\beq
(f_{1,0}^{-1}+\rho_1)\ln\frac{1}{\epsilon}.
\eeq

\textbf{(b)} (Recovery process: see Figure \ref{phase3type}) Next, we consider the case when there are three distinct trait types, 
 $\X=\{x_0,x_1,x_2\}$. At the same time as the population  on the  site  $x_0$ migrates towards 
 the new  site $x_1$ (as shown in (a)), the population of trait $x_1$ can migrate to the site
 $x_2$. Let $\xi_t^{\epsilon}(x_2):=\langle X_t^{\epsilon}, 1_{\{x_2\}}\rangle$. In the following, we re-analyze the evolution process after adding one more trait $x_2$ to the previous case with trait space $\{x_0, x_1\}$. 
Due to an $\epsilon$-fraction of initial migration from the subpopulation $x_0$ to $x_1$  it follows that 
a $\epsilon^2$-fraction migrates from subpopulation $x_1$ to $x_2$.
We have $\xi_{T^{\epsilon, 1}}^{\epsilon}(x_2)=\epsilon\xi_{T^{\epsilon, 1}}^{\epsilon}(x_1)=\epsilon^2\bar\xi(x_0)$.  
Since the population growth of trait $x_2$ is in an exponential rate $b(x_2)-d(x_2)$, the time needed 
for $\xi_t^{\epsilon}(x_2)$, starting with mass of order $\epsilon^2$ to reach the given $\eta$-level, is of order $\frac{2}{b(x_2)-d(x_2)}\ln\frac{1}{\epsilon}$.

As shown in Figure \ref{phase3type}, because of assumption (C3) we have that 
$\frac{2}{b(x_2)-d(x_2)}>\frac{1}{f_{1,0}}$, and the influence of the population at 
$x_2$ is negligible before the  time $T^{\eta,1}$ when the population at  $x_1$ reaches the  level $\eta$. 
Since $\widetilde T^{\eta,1}-T^{\eta,1}=O(1)$, the 
population at $x_2$, starting with $\xi_{\widetilde T^{\eta,1}}^{\epsilon}(x_2)=\epsilon\cdot O(1)$, evolves under the competition from its resident population $x_1$ as follows
\beq\label{x_2 equation}
\bea
\dot {\xi}_t^{\epsilon}(x_2)=&\big(b(x_2)-d(x_2)-\al(x_2,x_1)\xi_t^{\epsilon}(x_1)-\al(x_2,x_2)\xi_t^{\epsilon}(x_2)\big)\xi_t^{\epsilon}(x_2)\\
&-\epsilon\xi_t^{\epsilon}(x_2)b(x_2)m(x_2,x_1)+\epsilon\xi_t^{\epsilon}(x_1)b(x_1)m(x_1,x_2),
\eea
\eeq
where $\xi_{\widetilde T^{\eta,1}}^{\epsilon}(x_1)\in(\bar\xi(x_1)-\eta, \bar\xi(x_1)+\eta)$. On the other hand, until time $\widetilde T^{\eta,1}$ the populations at $x_0$ and $x_1$ still behave the same as in Step 1-Step 4.
Thus, we embed Figure \ref{dynamics_2type} into Figure \ref{phase3type} and continue the proof based on the four-step analysis in (a).

 \begin{figure}[hbtp]
 \centering
 \def\svgwidth{350pt}
 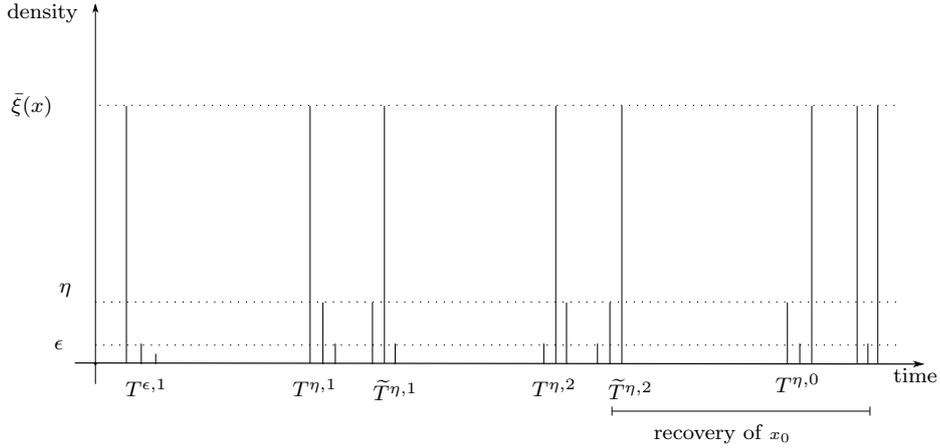
 \caption{{\footnotesize Three-type density evolution ($\xi_t^{\epsilon}(x_0), \xi_t^{\epsilon}(x_1), \xi_t^{\epsilon}(x_2)$)}}.
 \label{phase3type}
 \end{figure}

Let $T^{\eta,2}$ be the first time when $\xi_t^{\epsilon}(x_2)$ enters above the level $\eta$.  
By similar arguments as used in Step 2, one can control $\xi_t^{\epsilon}(x_2)$ by two other curves $\check{\xi}_t(x_2)\leq\xi_t^{\epsilon}(x_2)\leq\hat{\xi}_t(x_2)$ described as follows, for $\widetilde T^{\eta,1}<t<T^{\eta,2}$,
\beq
\dot{\check {\xi}}_t(x_2)=(f_{2,1}-\check C\eta)\check {\xi}_t(x_2),
\eeq
and
\beq
\dot{\hat {\xi}}_t(x_2)=(f_{2,1}+\hat C\eta)\hat {\xi}_t(x_2),
\eeq
where the constants $\check C, \hat C$ change from line to line and $\check{\xi}_{\widetilde T^{\eta,1}}(x_2)=\hat{\xi}_{\widetilde T^{\eta,1}}(x_2)=\epsilon\cdot O(1)$.
It follows that
\beq\label{3typeproof_time estimate}
\frac{1}{f_{2,1}+\hat C\eta}\ln \frac{1}{\epsilon}<T^{\eta,2}-\widetilde T^{\eta,1}<\frac{1}{f_{2,1}-\check C\eta}\ln \frac{1}{\epsilon}.
\eeq
On other other hand, due to the comparison assumption (C3): $\frac{2}{b(x_2)-d(x_2)}>\frac{1}{f_{1,0}}+\frac{1}{f_{2,1}}$, we conclude that the shorter time length among the two for $\xi_t^{\epsilon}(x_2)$ to reach $\eta$-level satisfies
\beq
\left(\frac{1}{f_{1,0}}+\frac{1}{f_{2,1}}-\de\right)\ln\frac{1}{\epsilon}<T^{\eta,2}
 <\left(\frac{1}{f_{1,0}}+\frac{1}{f_{2,1}}+\de\right)\ln\frac{1}{\epsilon}.
\eeq

During time interval $[\widetilde T^{\eta,1}, T^{\eta,2}]$, consider the population $\xi_t^{\epsilon}(x_0)$. We inherit the estimate $\check\xi_t(x_0)<\xi_t^{\epsilon}(x_0)<\hat\xi_t(x_0)$ from Step 4. Subpopulations $\check\xi_t(x_0)$ and $\hat\xi_t(x_0)$ are described by solutions of two equations \eqref{Invasion_step4_lower process} and \eqref{Invasion_step4_upper process}, which imply that

\beq
\check\xi_t(x_0)=\check\xi_{\widetilde T^{\eta,1}}(x_0)e^{(f_{0,1}-\check C\eta)(t-\widetilde T^{\eta,1})},
\eeq
and
\beq
\hat\xi_t(x_0)=\hat\xi_{\widetilde T^{\eta,1}}(x_0)e^{(f_{0,1}+\hat C\eta)(t-\widetilde T^{\eta,1})}
\eeq
with
$\check{\xi}_{\widetilde T^{\eta,1}}
(x_0)=\hat{\xi}_{\widetilde T^ {\eta,1}}
(x_0)=\eta$.

Combining above with \eqref{3typeproof_time estimate}, one obtains
\beq
\eta\epsilon^{-\frac{f_{0,1}-\check C\eta}{f_{2,1}-\check C\eta}}<\check\xi_{T^{\eta,2}}(x_0)<\xi_{T^{\eta,2}}^{\epsilon}(x_0)
<\hat\xi_{T^{\eta,2}}(x_0)<\eta\epsilon^{-\frac{f_{0,1}+\hat{C}\eta}{f_{2,1}+\hat C\eta}}.
\eeq

Taking $\epsilon$-migration from its neighbor site $x_1$ into account, the mass on $x_0$ is of order $\epsilon^{\frac{|f_{0,1}|}{f_{2,1}}}\vee\epsilon$. Due to assumption $(C4):\frac{|f_{0,1}|}{f_{2,1}}<1$, one obtains that $\epsilon^{\frac{|f_{0,1}|}{f_{2,1}}}\vee\epsilon=\epsilon^{\frac{|f_{0,1}|}{f_{2,1}}}$.

Short after $T^{\eta,2}$, as $\epsilon$ tends to 0, there follows an immediate swap between $\xi_t^{\epsilon}(x_1)$ and $\xi_t^{\epsilon}(x_2)$ approximated by a Lotka-Volterra system as in Step 3. Denote 
 by $\widetilde T^{\eta,2}$ the first time when $(\xi_t^{\epsilon}(x_1),\xi_t^{\epsilon}(x_2))$ enters the 
$\eta$-neighborhood of the equilibrium $(0,\bar{\xi}(x_2))$, 
i.e. $\xi_{\widetilde T^{\eta,2}}^{\epsilon}(x_1)=\eta$. Also, $\widetilde T^{\eta,2}-T^{\eta,2}$ is of order $O(1)$. Thus, one obtains
\beq\label{3typeproof_rebirth initial estimate}
\epsilon^{-\frac{f_{0,1}-\check C\eta}{f_{2,1}-\check C\eta}}\cdot O(\eta)<\xi_{\widetilde T^{\eta,2}}^{\epsilon}(x_0)=\xi_{T^{\eta,2}}^{\epsilon}(x_0)\cdot O(1)<\epsilon^{-\frac{f_{0,1}+\hat{C}\eta}{f_{2,1}+\hat C\eta}}\cdot O(\eta).
\eeq
Let $T^{\eta,0}$ denote the first time  after time $\widetilde T^{\eta,2}$ when $\xi_t^{\epsilon}(x_0)$ 
reaches $\eta$-level. For $\widetilde T^{\eta,2}<t<T^{\eta,0}$, $\xi_t^{\epsilon}(x_0)$ 
is governed approximately by a logistic equation
\beq
\dot{\xi}_t(x_0)=(b(x_0)-d(x_0)-\al(x_0,x_0)\xi_t(x_0))\xi_t(x_0).
\eeq
Then, we have the differential inequality
\beq
(b(x_0)-d(x_0)-\al(x_0,x_0)\eta)\xi_t(x_0)<\dot{\xi}_t(x_0)<(b(x_0)-d(x_0))\xi_t(x_0),
\eeq
where $\xi_{\widetilde T^{\eta,2}}(x_0)$ satisfies \eqref{3typeproof_rebirth initial estimate}. Then, by Gronwall's inequality, one obtains
\beq\label{3typeproof_rebirth time}
\bea
-\frac{f_{0,1}+\hat{C}\eta}{(f_{2,1}+\hat C\eta)(b(x_0)-d(x_0))}\ln\frac{1}{\epsilon}&<T^{\eta,0}-\widetilde T^{\eta,2}\\
&<-\frac{f_{0,1}-\check C\eta}{(f_{2,1}-\check C\eta)(b(x_0)-d(x_0)-\al(x_0,x_0)\eta)}\ln\frac{1}{\epsilon}.
\eea
\eeq
After $T^{\eta,0}$, $\xi_t^{\epsilon}(x_0)$ approaches $\bar\xi(x_0)$ in time length of order 1.

Combining the analysis on $T^{\eta,1}$ in Step 2, and the estimates in \eqref{3typeproof_time estimate} and \eqref{3typeproof_rebirth time}, since  $\eta>0$ is arbitrary, we obtain that 
\beq\label{3typeproof_multi times}
\bea
&\lim\limits_{\eta\to 0}\frac{T^{\eta,1}}{\ln\frac{1}{\epsilon}}=\frac{1}{f_{1,0}}=:I_1,\\
&\lim\limits_{\eta\to 0}\frac{T^{\eta,2}-\widetilde T^{\eta,1}}{\ln\frac{1}{\epsilon}}=\frac{1}{f_{2,1}}=:I_2-I_1,\\
&\lim\limits_{\eta\to 0}\frac{T^{\eta,0}-\widetilde T^{\eta,2}}{\ln\frac{1}{\epsilon}}=\frac{-f_{0,1}}{f_{2,1}(b(x_0)-d(x_0))}=:S_0.
\eea
\eeq
Therefore, $X_t^{\epsilon}$, rescaled on a time scale of order $\ln\frac{1}{\epsilon}$, converges to the TST process $U_t$ $(L=2)$ with the form \eqref{tss_3type}.
\end{proof}

\begin{prop}\label{TSS_4TYPE}
Admit the same conditions as in Theorem \ref{TST_deterministic}. Consider the case when $L=3$, i.e. $\X=\{x_0, x_1, x_2, x_3\}$. Then the limit process $(U_t)_{t\geq 0}$ has the form
\beq\label{tss_4type}
U_t\equiv
\left\{
  \begin{array}{ll}
    \bar{\xi}(x_0)\de_{x_0} & \textrm{for}~ 0\leq t\leq I_1, \\
    \bar{\xi}(x_1)\de_{x_1} & \textrm{for}~ I_1< t\leq I_2, \\
    \bar{\xi}(x_2)\de_{x_2} & \textrm{for}~ I_2< t\leq I_3, \\
    \bar{\xi}(x_3)\de_{x_3} & \textrm{for}~ I_3< t\leq I_3+S_1,\\
     \bar{\xi}(x_1)\de_{x_1}+\bar{\xi}(x_3)\de_{x_3} & \textrm{for}~ t> I_3+S_1
  \end{array}
\right.
\eeq
where $I_3=\tfrac{1}{f_{1,0}}+\tfrac{1}{f_{2,1}}+\tfrac{1}{f_{3,2}}$, and $S_1=\frac{|f_{1,2}|}{f_{3,2}(b(x_1)-d(x_1))}$
\end{prop}

\begin{proof}
(See Figure \ref{phase4type}) 
Due to an $\epsilon$-fraction of migration from a subpopulation to its 
neighbor subpopulation it follows that $\xi_{T^{\epsilon,1}}^{\epsilon}(x_3)=\epsilon\xi_{T^{\epsilon,1}}^{\epsilon}(x_2)=\epsilon^2\xi_{T^{\epsilon,1}}^{\epsilon}(x_1)=\epsilon^3\bar\xi(x_0)$.
Since the density of trait type $x_3$ increases in an exponential speed $b(x_3)-d(x_3)$, the time 
needed to reach a level  $\eta$-level is of order $\frac{3}{b(x_3)-d(x_3)}\ln \frac{1}{\epsilon}$.
By assumption (C3): $\frac{3}{b(x_3)-d(x_3)}>\frac{1}{f_{1,0}}+\frac{1}{f_{2,1}}$, it implies that the 
population on trait site $x_3$ is negligible, i.e. of order $\epsilon$, before the time  $T^{\eta,2}$. 
Thus until that time the evolution of the populations at the other sites preoceeds as if this site did not exist.
 The  analysis and notations such as $T^{\eta,1}, \widetilde T^{\eta,1}$, $T^{\eta,2}, \widetilde T^{\eta,2}$ then 
carry over from the proof of Proposition \ref{TSS_3TYPE}.

 \begin{figure}[hbtp]
 \centering
 \def\svgwidth{350pt}
 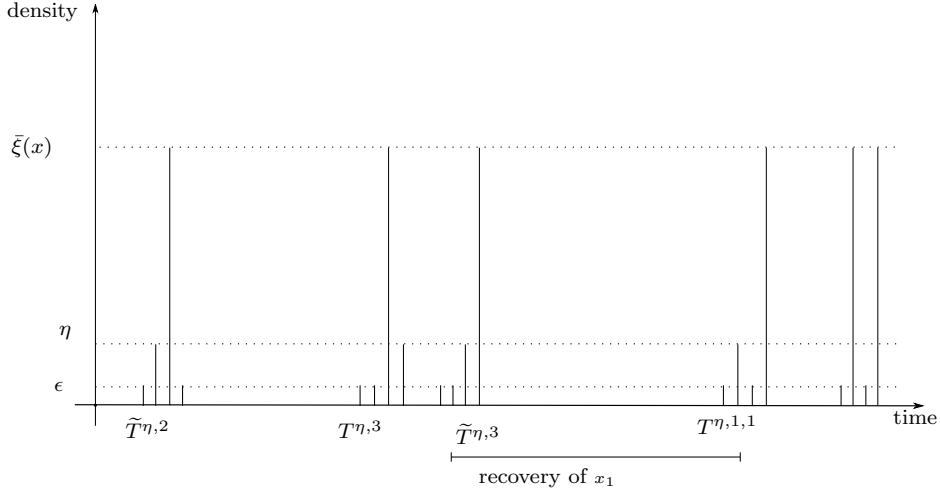
 \caption{{\footnotesize Four-type density evolution ($\xi_t^{\epsilon}(x_0), \xi_t^{\epsilon}(x_1), \xi_t^{\epsilon}(x_2), \xi_t^{\epsilon}(x_3)$) }}.
 \label{phase4type}
 \end{figure}

The evolution of $ \dot {\xi}_t^{\epsilon}(x_3)$ after time $\widetilde T^{\eta,2}$ is governed by the equation
\beq
\bea
\dot {\xi}_t^{\epsilon}(x_3)=&\big(b(x_3)-d(x_3)-\al(x_3,x_2)\xi_t^{\epsilon}(x_2)-\al(x_3,x_3)\xi_t^{\epsilon}(x_3)\big)\xi_t^{\epsilon}(x_3)\\
&-\epsilon\xi_t^{\epsilon}(x_3)b(x_3)m(x_3,x_2)+\epsilon\xi_t^{\epsilon}(x_2)b(x_2)m(x_2,x_3),
\eea
\eeq
with initial value $\xi_{\widetilde T^{\eta,2}}^{\epsilon}(x_3)=\epsilon\cdot O(1)$.

Let $T^{\eta,0}$ and $T^{\eta,3}$ be the first time (resp.) for $\xi_t^{\epsilon}(x_0)$ and $\xi_t^{\epsilon}(x_3)$ to reach $\eta$-level after $\widetilde T^{\eta,2}$.  Similarly as in the derivation of Eq. \eqref{3typeproof_time estimate}, we get
\beq\label{4ypeproof_estimate eta level}
T^{\eta,3}-\widetilde T^{\eta,2}\sim \frac{1}{f_{3,2}}\ln\frac{1}{\epsilon}
\eeq
where  $f(\epsilon)\sim g(\epsilon)$ means $\lim\limits_{\epsilon\to 0}f(\epsilon)/g(\epsilon)=1$.

Recall from \eqref{3typeproof_rebirth time} that
\beq\label{4typeproof_x0 rebirth}
T^{\eta,0}-\widetilde T^{\eta,2}\sim \frac{-f_{0,1}}{f_{2,1}(b(x_0)-d(x_0))}\ln\frac{1}{\epsilon}.
\eeq
From assumption (C4), one obtains that
\beq\label{4typeproof_rebirth time comparison}
\frac{-f_{0,1}}{f_{2,1}(b(x_0)-d(x_0))}-\frac{1}{f_{3,2}}>\frac{-f_{1,2}}{f_{3,2}(b(x_1)-d(x_1))}>0,
\eeq
which implies that $T^{\eta,3} < T^{\eta,0}$. 
Hence, for $t\in[\widetilde T^{\eta,2}, T^{\eta,3}]$, the population at site $x_0$, $\xi_t^{\epsilon}(x_0)$, 
stays in some small $\epsilon$-dependent   neighborhood of 0. Furthermore, $\xi_t^{\epsilon}(x_1)$ is influenced mainly from the competition with  $\bar\xi(x_2)$. Using comparison arguments as above,  
one derives that at time $T^{\eta,3}$
\beq
\bea
\xi_{T^{\eta,3}}^{\epsilon}(x_1)&=\xi_{\widetilde T^{\eta,2}}^{\epsilon}(x_1)e^{f_{1,2}(T^{\eta,3}-\widetilde T^{\eta,2})}\\
&\sim\eta e^{f_{1,2}f_{3,2}^{-1}\ln\frac{1}{\epsilon}}\\
&\sim\eta\epsilon^{-\frac{f_{1,2}}{f_{3,2}}}.
\eea
\eeq
Similarly as in Step 3,  after the time $T^{\eta,3}$, $\xi_t^{\epsilon}(x_2)$ and $\xi_t^{\epsilon}(x_3)$ 
swap their mass   in a time of order 1, and $\xi_t^{\epsilon}(x_2)$ decreases below  a level  $\eta$ at time $\widetilde T^{\eta,3}$. After time $\widetilde T^{\eta,3}$, $\xi_t^{\epsilon}(x_1)$ evolves approximately as a logistic growth curve since there is only negligible competition from neighbors $x_0, x_2$, i.e.
\beq
\dot{\xi}_t^{\epsilon}(x_1)=(b(x_1)-d(x_1)-\al(x_1,x_1)\xi_t^{\epsilon}(x_1))\xi_t^{\epsilon}(x_1),
\eeq
with initial value $\xi_{\widetilde T^{\eta,3}}^{\epsilon}(x_1)\sim\eta\epsilon^{-\frac{f_{1,2}}{f_{3,2}}}$.

Denote by $T^{\eta,1,1}$ the second time for the population at site  $x_1$ to increase to 
 the level $\eta$. Because of the exponential growth property, as in the derivation of 
 \eqref{4typeproof_x0 rebirth}, one obtains
\beq\label{4ypeproof_x1 rebirth}
T^{\eta,1,1}-\widetilde T^{\eta,3}\sim\frac{-f_{1,2}}{f_{3,2}(b(x_1)-d(x_1))}\ln\frac{1}{\epsilon}.
\eeq
Comparing the times computed above  with the estimate \eqref{4typeproof_rebirth time comparison} on 
$T^{\eta,0}$, one observes that $T^{\eta,1,1}<T^{\eta,0}$.  This means that $\xi_t^{\epsilon}(x_1)$ 
recovers to level  $\eta$ faster than $\xi_t^{\epsilon}(x_0)$ did. Consequently, $\xi_t^{\epsilon}(x_0)$ will 
be pushed  to 0 due to competition from the fitter type $x_1$.

Combining \eqref{4ypeproof_estimate eta level}, \eqref{4ypeproof_x1 rebirth} and the first two equations in 
\eqref{3typeproof_multi times}, we obtain the claimed form of  the TST limiting configuration $(U_t)$ for $L=3$.
\end{proof}

\begin{lem}\label{Lem:C4}
Assumption (C4) implies the following inequalities, for any $4\leq L\in\N$
\beq\label{assumption C4_1}
\bea
&\frac{-f_{0,1}}{f_{2,1}(b(x_0)-d(x_0))}>\frac{1}{f_{3,2}}+\cdots+\frac{1}{f_{L,L-1}},\\
&\frac{-f_{1,2}}{f_{3,2}(b(x_1)-d(x_1))}>\frac{1}{f_{4,3}}+\cdots+\frac{1}{f_{L,L-1}},\\
\vdots\\
&\frac{-f_{L-3,L-2}}{f_{L-1,L-2}(b(x_{L-3})-d(x_{L-3}))}>\frac{1}{f_{L,L-1}}
\eea
\eeq
and
\beq\label{assumption C4_2}
\bea
&\frac{-f_{L-4,L-3}}{f_{L-2,L-3}(b(x_{L-4})-d(x_{L-4}))}-\frac{1}{f_{L-1,L-2}}-\frac{1}{f_{L,L-1}}
>\frac{-f_{L-2,L-1}}{f_{L,L-1}(b(x_{L-2})-d(x_{L-2}))},\\
\vdots\\
&\frac{-f_{0,1}}{f_{2,1}(b(x_{0})-d(x_{0}))}-\frac{1}{f_{3,2}}-\frac{1}{f_{4,3}}
>\frac{-f_{2,3}}{f_{4,3}(b(x_{2})-d(x_{2}))}
\eea
\eeq
and so on.
\end{lem}
The proof of this Lemma follows iterations straightforward from assumption (C4). On the one hand, from \eqref{assumption C4_1}, it implies that when it passes to the limit process $U_t$, all processes except $\xi_t^{\epsilon}(x_L)$ stay in $\epsilon$-dependent infinitesimal neighborhoods of 0 at time $T^{\eta,L}$ which denotes the establishing time for type $x_L$. It leads to monomorphic transportation of the mass from the initial trait $x_0$ to the fittest trait $x_L$ in the first half period. On the other hand, from \eqref{assumption C4_2}, it guarantees that the fitter one recovers earlier than the unfit traits alternatively backwards to the most unfit one.

\begin{proof}[Proof of Theorem \ref{TST_deterministic}]
After justifying the form of $(U_t)_{t\geq 0}$ for $L=2$ in Proposition \ref{TSS_3TYPE} and $L=3$ in Proposition \ref{TSS_4TYPE}, we proceed the proof along two lines, according to whether $L$ is an even or odd integer. We here only prove cases along the line when $L$ is a even integer.

We now deduce the expression for $(U_t)_{t\geq 0}$ when $L=4$ based on the result for $L=2$ (see Remark \ref{Remark:TST_finite} (2) and Figure \ref{treeconstruction}). 

Based on the analysis in the proof of Proposition \ref{TSS_4TYPE}, after time $\widetilde T^{\eta,3}$, 
we introduce $T^{\eta,4}$ which is defined as the first time for $\xi_t^{\epsilon}(x_4)$ to reach the $\eta$-level. Similarly as before, we can show that
\beq
T^{\eta,4}-\widetilde T^{\eta,3}\sim \frac{1}{f_{4,3}}\ln\frac{1}{\epsilon}.
\eeq

Then, after time $T^{\eta,4}$, to mimic Step 3 in the proof of Proposition \ref{TSS_3TYPE}, it follows with a selective sweep between subpopulation $x_3§$ and $x_4$
until time  $\widetilde T^{\eta,4}$ such that $\xi_{\widetilde T^{\eta,4}}^{\epsilon}(x_3)=\eta$. 

At this stage, 
 $\xi_t^{\epsilon}(x_2)$ starts to recover due to the lack of competition from 
$\xi_t^{\epsilon}(x_3)$. Similarly as in the derivation of \eqref{4ypeproof_x1 rebirth}, the time  needed for $\xi_t^{\epsilon}(x_2)$ to again reach 
the $\eta$-level (denoted by $T^{\eta,2,2}$) can be computed explicitly
\beq\label{Proof:TSS_finite-X_2recovertime}
T^{\eta,2,2}-\widetilde T^{\eta,4}\sim\frac{-f_{2,3}}{f_{4,3}(b(x_2)-d(x_2))}\ln\frac{1}{\epsilon}=:S_2\ln\frac{1}{\epsilon}
\eeq
which, due to $\widetilde T^{\eta,4}\sim I_4\ln\frac{1}{\epsilon}$, implies
\beq
T^{\eta,2,2}\sim (I_4+S_2)\ln\frac{1}{\epsilon}.
\eeq
After that, it will approach its equilibrium $\bar\xi(x_2)$ as a solution of a logistic equation.
Consequently, $\xi_t^{\epsilon}(x_1)$ will drift to 0 due to competition from its fitter neighbor-trait $x_2$.

Recall from \eqref{4typeproof_x0 rebirth} that
\beq\label{Proof:TSS_finite-X_0recovertime}
\bea
T^{\eta,0}-\widetilde T^{\eta,4}&
                                  \sim\Big[\frac{-f_{0,1}}{f_{2,1}(b(x_0)-d(x_0))}-\frac{1}{f_{3,2}}-\frac{1}{f_{4,3}}\Big]\ln\frac{1}{\epsilon}\\
                                &=:(S_0+I_2-I_4)\ln\frac{1}{\epsilon}
\eea
\eeq
which implies
\beq
T^{\eta,0}\sim (I_2+S_0)\ln\frac{1}{\epsilon}.
\eeq
Combining the above two estimates \eqref{Proof:TSS_finite-X_2recovertime} and \eqref{Proof:TSS_finite-X_0recovertime} with assumption \eqref{assumption C4_2}, one observes that
\beq
T^{\eta,2,2}-\widetilde T^{\eta,4}<T^{\eta,0}-\widetilde T^{\eta,4}.
\eeq
Moreover, we have
\beq
\bea
\lim\limits_{\eta\to 0}\frac{T^{\eta,2,2}}{\ln\frac{1}{\epsilon}}&=\frac{1}{f_{1,0}}+\frac{1}{f_{2,1}}+\frac{1}{f_{3,2}}+\frac{1}{f_{4,3}}
+\frac{|f_{2,3}|}{f_{4,3}(b(x_2)-d(x_2))}\\
&=I_4+S_2,
\eea
\eeq
and
\beq
\lim\limits_{\eta\to 0}\frac{T^{\eta,0}}{\ln\frac{1}{\epsilon}}=I_2+S_0.
\eeq
We thus obtain the explicit form of \eqref{tss_U_t_even} for $L=4$
\beq\label{tss_5type}
U_t\equiv
\left\{
  \begin{array}{ll}
    \bar{\xi}(x_0)\de_{x_0} & \textrm{for}~ 0\leq t\leq I_1, \\
    \bar{\xi}(x_1)\de_{x_1} & \textrm{for}~ I_1< t\leq I_2, \\
    \bar{\xi}(x_2)\de_{x_2} & \textrm{for}~ I_2< t\leq I_3, \\
    \bar{\xi}(x_3)\de_{x_3} & \textrm{for}~ I_3< t\leq I_4, \\
    \bar{\xi}(x_4)\de_{x_4} & \textrm{for}~ I_4< t\leq I_4+S_2,\\
    \bar{\xi}(x_2)\de_{x_2}+\bar{\xi}(x_4)\de_{x_4} & \textrm{for}~I_4+S_2< t\leq I_2+S_0,\\
     \bar{\xi}(x_0)\de_{x_0}+\bar{\xi}(x_2)\de_{x_2}+\bar{\xi}(x_4)\de_{x_4} & \textrm{for}~ t> I_2+S_0
  \end{array}
\right.
\eeq
which is constructed on top of $U_t$ for $L=2$ (see \eqref{tss_3type}) by partitioning the interval $(I_2, I_2+S_0]$ into intervals $(I_2, I_3]\cup (I_3, I_4]\cup (I_4, I_4+S_2]\cup (I_4+S_2, I_2+S_0]$.

To mimic a similar procedure, a TST process for $L=6$ can be obtained by connecting the TST process for $L=4$ with a sub-TST consisting of traits $\{x_4, x_5, x_6\}$ specified as in Proposition \ref{TSS_3TYPE}. 

Recursively, for all even integers $L$, the form of \eqref{tss_U_t_even} follows. 
\end{proof}

\subsection{Proof of Theorem \ref{TST_theorem_infinite trait}}\label{section 2.5.2}

To prove the Theorem \ref{TST_theorem_infinite trait}, we proceeds
by listing two key lemmas. In the first one, we conclude that the occurrence time of each successive mutation is asymptotically characterized by exponential distribution
while rescaling time in an appropriate way. 
In the second lemma, we justify that a one-step transition from a current configuration to a new one is in probability one as the migration rate $\epsilon$ tends to 0. To avoid repeating arguments, we here only prove the first case of Definition \ref{TST_Definition_infinite tree} while the second one follows a same fashion. The proof is carried out by the method of mathematical induction.

For any non-negative integer $l$, denote by $\Gamma^{(2l)}$ the atomic measure with finite support, i.e., $\Gamma^{(2l)}=\sum_{i=0}^l\bar\xi(x_{2i}^{(2l)})\delta_{x_{2i}^{(2l)}}$. Similarly, set $\Gamma^{(2l+1)}=\sum_{i=1}^{l+1}\bar\xi(x_{2i-1}^{(2l+1)})\delta_{x_{2i-1}^{(2l+1)}}$, whose form is described as in Definition \ref{TST_Definition_infinite tree}.
As for the transition from $\Gamma^{(0)}$ to $\Gamma^{(1)}$, it is trivial to be proved as in Proposition \ref{TSS_3TYPE} (a). To the end, it remains to show that the transition rule also holds from configuration $\Gamma^{(2l)}$ to $\Gamma^{(2l+1)}$ for any $l\in\N$. Denote by $\PP_{\Gamma^{(2l)}}$ the law of the process $X^{\epsilon,\sigma}_{\cdot}$ with initial configuration $\Gamma^{(2l)}$.
Denote by $\tau^{\epsilon}$ the first time after 0 when there occurs a new mutation event.

\begin{lem}\label{Lemma_exp time in TST}
Admit the same conditions as in Theorem \ref{TST_theorem_infinite trait}.
\beq
\lim\limits_{\epsilon\to 0} \PP_{\Gamma^{(2l)}}\big(\tau^{\epsilon}>\frac{t}{\sigma}\big)=\exp\Big(-t\sum_{i=0}^l\bar\xi(x_{2i}^{(2l)})\mu(x_{2i}^{(2l)})\Big).
\eeq
\end{lem}
This lemma can be proved in a very similar way as the one for \cite[Lemma 2 (c)]{Cha06}. We will not repeat the details here.

\begin{lem}\label{Lemma_new config. TST}
Assume $X_0^{\epsilon,\sigma}=\Gamma^{(2l)}+\rho\delta_{x_{2k}^{(2l)}+h}$. Then, for any $\varepsilon>0$, there exists a constant $C>0$ such that
\beq
\lim\limits_{\epsilon\to 0}\PP\big(\tau^{\epsilon}>\ln\frac{1}{\epsilon}, \, \sup\limits_{t\in(C\ln\frac{1}{\epsilon}, \tau^{\epsilon})}\| X_t^{\epsilon, \sigma}-\Gamma^{(2l+1)}\|<\varepsilon\big)=1
\eeq
where $\Gamma^{(2l+1)}$ is defined as in Definition \ref{TST_Definition_infinite tree} (i), and $\|\cdot\|$ is the total variation norm.
\end{lem}
\begin{proof}
From Lemma \ref{Lemma_exp time in TST} and $\frac{1}{\sigma}\gg\ln\frac{1}{\epsilon}$, one concludes that, for any $C>0$,
\[
\lim\limits_{\epsilon\to 0}\PP\big(\tau^{\epsilon}>C\ln\frac{1}{\epsilon}\big)=1.
\]
According to assumption (D1), there will be one and only one ranked position for the new trait $x_{2k}^{(2l)}+h$ among the supporting traits $\{x_{0}^{(2l)}, \ldots, x_{2j}^{(2l)}, \ldots, x_{2l}^{(2l)}\}$ of $\Gamma^{(2l)}$. We will classify two cases according to whether $x_{2k}^{(2l)}+h$ falls on right or left hand side of some $x_{2j}^{(2l)}$ for $0\leq j\le l$.

If there exists some $x_{2j}^{(2l)}$ ($0\leq j\leq l$) such that $x_{2k}^{(2l)}+h$ falls on the right of $x_{2j}^{(2l)}$, one has the local fitness order
\beq
\ldots\prec x_{2j-1}^{(2l)}\prec x_{2j}^{(2l)}\prec x_{2k}^{(2l)}+h\prec x_{2j+1}^{(2l)}\prec\ldots.
\eeq

Since it is unpopulated for both trait sites $x_{2j-1}^{(2l)}$ and $x_{2j+1}^{(2l)}$ at state $\Gamma^{(2l)}$, we consider $\big(x_{2j}^{(2l)},\, x_{2k}^{(2l)}+h\big)$ as an isolated two-type system because of the assumption of nearest-neighbor competition. By analysis in Lemma \ref{lemma_noncoexistence condition}, 
the population density of this two-type system converges to $\big(0,\,\bar\xi(x_{2k}^{(2l)}+h)\big)$ 
in time of order $O\left(\ln\frac{1}{\epsilon}\right)$ as $\epsilon\to 0$. On the right hand side of this 
pair $\big(x_{2j}^{(2l)},\, x_{2k}^{(2l)}+h\big)$, the configuration 
$\sum_{i=j+1}^l\bar\xi(x^{(2l)}_{2i})\delta_{x^{(2l)}_{2i}}$ keeps stable as at previous state 
$\Gamma^{(2l)}$. Whereas on the left hand side of $x_{2j}^{(2l)}$, population density of trait $x_{2j-1}^{(2l)}$,
 due to the decay of its competitive fitter neighbor trait $x_{2j}^{(2l)}$, recovers exponentially fast upto the 
stable equilibrium $\bar\xi(x^{(2l)}_{2j-1})$ of the following logistic equation
\beq
\dot{\xi}_t^{\epsilon}(x_{2j-1}^{(2l)})=(b(x_{2j-1}^{(2l)})-d(x_{2j-1}^{(2l)})-
\al(x_{2j-1}^{(2l)}, x_{2j-1}^{(2l)})\xi_t^{\epsilon}(x_{2j-1}^{(2l)}))\xi_t^{\epsilon}(x_{2j-1}^{(2l)}).
\eeq
Continuing in the same way, the mass occupation flips on the left hand side of the trait site $x_{
2j}^{(2l)}$ such that traits $\{x_{2i-1}^{(2l)}, \,1\leq i\leq j\}$ get re-established while 
$\{x_{2i}^{(2l)}, \,0\leq i\leq j\}$ are eliminated. By 
similar arguments as in the finite trait case (see Section \ref{section 2.5.1}), the 
entire rearrangement process can be completed in time of order $O(\ln\frac{1}{\epsilon})$. 
We obtain the new equilibrium configuration
\begin{equation}
\Gamma^{(2l+1)}=\sum_{i=1}^j\bar\xi(x^{(2l)}_{2i-1})\delta_{x^{(2l)}_{2i-1}}+\bar\xi(x_{2k}^{(2l)}+h)\de_{x_{2k}^{(2l)}+h}+\sum_{i=j+1}^l\bar\xi(x^{(2l)}_{2i})\delta_{x^{(2l)}_{2i}}.
\end{equation}
In the other case, the fitness location of trait $x^{(2l)}_{2k}+h$ falls on the left hand side of some $x^{(2l)}_{2j}$ for $0\leq j\leq l$, that is,
\[
x_{2j-1}^{(2l)}\prec x_{2k}^{(2l)}+h\prec x_{2j}^{(2l)}\prec x_{2j+1}^{(2l)}.
\]
Similarly, consider the sub-populations $\big(x_{2k}^{(2l)}+h,\, x_{2j}^{(2l)}\big)$ as an isolated dimorphic system as in Lemma \ref{lemma_noncoexistence condition}. Since $\rho\to 0$ as $\epsilon\to 0$, we obtain that $\left(\xi_t(x_{2k}^{(2l)}+h), \xi_t(x_{2j}^{(2l)})\right)_{t\geq 0}$, starting with  $\left(\rho, \bar\xi(x^{(2l)}_{2j})\right)$, converges to $\left(0, \bar\xi(x^{(2l)}_{2j})\right)$. Consequently, due to a lack of competition from its nearest fitter neighbor $x_{2k}^{(2l)}+h$, sub-population $x^{(2l)}_{2j-1}$ starts to recover, so does $x^{(2l)}_{2i-1}$ for every $1\leq i<j$. Thus, we obtain the new equilibrium configuration
\[
\Gamma^{(2l+1)}=\sum_{i=1}^j\bar\xi(x^{(2l)}_{2i-1})\delta_{x^{(2l)}_{2i-1}}+\sum_{i=j}^l\bar\xi(x^{(2l)}_{2i})\delta_{x^{(2l)}_{2i}}.
\]
 In conclusion, the new configuration $\Gamma^{(2l+1)}$ is obtained by relabeling the traits as done in 
Definition \ref{TST_Definition_infinite tree} (i).
\end{proof}

Lemma \ref{Lemma_exp time in TST}, shows that the mutation occurs on the time scale 
$O(\frac{1}{\sigma})$. Recall from Section \ref{section 2.3} that the time scale for fixation is 
$O\left(\ln\frac{1}{\epsilon}\right)$ on a finite trait space. Combining them with the time scale 
separation constraint $\frac{1}{\sigma}\gg\ln\frac{1}{\epsilon}$ 
(heuristically introduced by Metz et al \cite{MGMJ96} and mathematically developed by 
Champagnat \cite{Cha06}), the proof of Theorem \ref{TST_theorem_infinite trait} is as follows.

\begin{proof}[Proof of Theorem \ref{TST_theorem_infinite trait}]
For any non-negative $L\in\N$, let $B^{(L)}$ be a measurable subset of $\X^{L}$, such that $B^{(L)}:=\{\texttt{x}^{(L)}=(x^{(L)}_0, \ldots, x^{(L)}_L): x^{(L)}_0\prec\ldots\prec x^{(L)}_L\}$ as in Assumption (D1). 
 Define, for any $L=2l$ even,
\beq\label{Proof:mainTheorem_config.ForEvenNumber}
\bar n(x^{(L)}_i)=
  \begin{cases}
    \bar{\xi}(x^{(L)}_i) & \textrm{for}~ i ~ \textrm{even~ integer} \\
      0                  & \textrm{for}~ i ~ \textrm{odd~ integer},
  \end{cases}
\eeq
and for $L=2l+1$ odd,
\beq\label{Proof:mainTheorem_config.ForOddNumber}
\bar n(x^{(L)}_i)=
  \begin{cases}
     0                    & \textrm{for}~ i ~ \textrm{even ~integer} \\
     \bar{\xi}(x^{(L)}_i) & \textrm{for}~ i ~ \textrm{odd ~integer}.
  \end{cases}
\eeq
Then the support process of $(\Gamma_t)_{t\geq 0}$ in Definition \ref{TST_Definition_infinite tree}, 
denoted by $(Z_t)_{t\geq 0}$, has the following infinitesimal generator
\beq
G \varphi (\texttt{x}^{(L)})= \int_{\R^d}\left(\varphi(\texttt{x}^{(L+1)})-\varphi (\texttt{x}^{(L)})\right)\beta(\texttt{x}^{(L)})\kappa(\texttt{x}^{(L)}, dh),
\eeq
where $\beta(\texttt{x}^{(L)})=\sum\limits_{i=0}^L \bar n(x^{(L)}_i)\mu(x^{(L)}_i)$, and the probability kernel
\beq\label{Proof:mainTheorem_probabKernel}
\kappa(\texttt{x}^{(L)}, dh)=\sum\limits_{k=0}^L \frac{\bar n(x^{(L)}_k)\mu(x^{(L)}_k)}{\beta(\texttt{x}^{(L)})}p(x^{(L)}_k, dh),
\eeq
and $\texttt{x}^{(L+1)}=(x^{(L+1)}_1, \ldots, x^{(L+1)}_{L+1})$ is determined by $\texttt{x}^{(L)}$ and $p(x^{(L)}_k, dh)$ as in Definition \ref{TST_Definition_infinite tree}.

According to Definition \ref{TST_Definition_infinite tree}, let $P_{\texttt{x}^{(L)}}$ be the law of $(Z_t)_{t\geq 0}$ with initial state $\texttt{x}^{(L)}$. 
Denote by $(S_n)_{n\geq 1}$ the sequence of occurrence times of mutations. By applying the 
strong Markov property at $S_1$, we obtain
\beq\label{Proof:mainTheorem_Recursion}
\bea
  &P_{\texttt{x}^{(L)}}
  \left( S_n< t< S_{n+1}, Z_t\in B^{(L+n)}\right)\\
  &=\int_0^t \beta(\texttt{x}^{(L)}) \exp^{-s\beta(\texttt{x}^{(L)})} \int_{h\in\R^d}P_{\texttt{x}^{(L+1)}}\left(S_{n-1}< t-s< S_n, Z_{t-s}\in B^{(L+n)}\right)\kappa(\texttt{x}^{(L)}, dh)ds.
\eea
\eeq
In particular,
\beq\label{Proof:mainTheorem_RecursionFirstStep}
P_{\texttt{x}^{(L)}}\left( 0\leq t< S_1, Z_t\in B^{(L)}\right)=1_{\{\texttt{x}^{(L)}\in B^{(L)}\}} \exp\left(-t\beta(\texttt{x}^{(L)})\right).
\eeq
The idea of our proof for the theorem is to show that the same relation as above holds for the  
rescaled processes $(X_{t/\sigma}^{\epsilon,\sigma})_{t\geq 0}$ as taking $\epsilon\to 0$ when we replace $Z_t$ by $Supp(X_{t/\sigma}^{\epsilon,\sigma})$ and replace $S_n$ by the $n$-th jump time $\tau^{\epsilon}_n$ of $(X_{t/\sigma}^{\epsilon,\sigma})_{t\geq 0}$.     

Let $\PP_{\texttt{x}^{(L)}}$ be the law of $(X_{t/\sigma}^{\epsilon,\sigma})_{t\geq 0}$ with initial state's support $\texttt{x}^{(L)}$. Consider the quantity
\beq
\PP_{\texttt{x}^{(L)}}
  \left(\tau^{\epsilon}_n< t<\tau^{\epsilon}_{n+1}, ~Supp(X_{t/\sigma}^{\epsilon,\sigma})\in B^{(L+n)}\right).
\eeq
For $n=0$, it is implied from Lemma \ref{Lemma_exp time in TST} that
\beq\label{Proof:mainTheorem_RecursionAsymptoticFirstStep}
\lim\limits_{\epsilon\to 0}\PP_{\texttt{x}^{(L)}}
  \left(0\leq t<\tau^{\epsilon}_{1}, ~Supp(X_{t/\sigma}^{\epsilon,\sigma})\in B^{(L)}\right)=1_{\{\texttt{x}^{(L)}\in B^{(L)}\}} \exp\left(-t\beta(\texttt{x}^{(L)})\right).
\eeq

For $n\geq 1$, by the strong Markov property at $\tau^{\epsilon}_1$, we obtain
\beq\label{Proof:mainTheorem_RecursionDeduction}
\bea
  &\PP_{\texttt{x}^{(L)}}
  \left(\tau^{\epsilon}_n< t<\tau^{\epsilon}_{n+1}, ~Supp(X_{t/\sigma}^{\epsilon,\sigma})\in B^{(L+n)}\right)\\
  &=\PP_{\texttt{x}^{(L)}}
  \left(\tau^{\epsilon}_1 < t\right) \int_{h\in\R^d}\PP_{\texttt{x}^{(L+1)}}\left(\tau^{\epsilon}_{n-1}< t-\tau^{\epsilon}_1< \tau^{\epsilon}_n, ~Supp(X_{\tfrac{t-\tau^{\epsilon}_1}{\sigma}}^{\epsilon,\sigma})\in B^{(L+n)}\right)\kappa^\epsilon(\texttt{x}^{(L)}, dh)
\eea
\eeq
where $\kappa^\epsilon(\texttt{x}^{(L)}, dh)=\sum\limits_{k=0}^L \tfrac{\xi^{\epsilon}_{\tau^{\epsilon}_1}(x^{(L)}_k)\mu(x^{(L)}_k)p(x^{(L)}_k, dh)}{\sum\limits_{k=0}^L \xi^{\epsilon}_{\tau^{\epsilon}_1}(x^{(L)}_k)\mu(x^{(L)}_k)}$ converges to $\sum\limits_{k=0}^L \tfrac{\bar n(x^{(L)}_k)\mu(x^{(L)}_k)p(x^{(L)}_k, dh)}{\sum\limits_{k=0}^L \bar n(x^{(L)}_k)\mu(x^{(L)}_k)}$ as $\epsilon\to 0$ due to Theorem \ref{TST_deterministic} and the form of $\bar n(\cdot)$ given in Eq.\eqref{Proof:mainTheorem_config.ForEvenNumber} and \eqref{Proof:mainTheorem_config.ForOddNumber}.

Substituting the terms on the RHS of Eq.\eqref{Proof:mainTheorem_RecursionDeduction} by their limits when taking $\epsilon\to 0$, combining with \eqref{Proof:mainTheorem_RecursionAsymptoticFirstStep} and \eqref{Proof:mainTheorem_probabKernel}, we obtain
\beq\label{Proof:mainTheorem_RecursionAsymptotic}
\bea
  &\lim\limits_{\epsilon\to 0}\PP_{\texttt{x}^{(L)}}
  \left(\tau^{\epsilon}_n< t<\tau^{\epsilon}_{n+1}, Supp(X_{t/\sigma}^{\epsilon,\sigma})\in B^{(L+n)}\right)\\
  &=\lim\limits_{\epsilon\to 0}\int_0^t \beta(\texttt{x}^{(L)})\exp\left(-s\beta(\texttt{x}^{(L)})\right) \int_{h\in\R^d} \sum\limits_{k=0}^L \tfrac{\bar  n(x^{(L)}_k)\mu(x^{(L)}_k)p(x^{(L)}_k, dh)}{\sum\limits_{k=0}^L \bar n(x^{(L)}_k)\mu(x^{(L)}_k)}\\
     &\qquad\quad \cdot \PP_{\texttt{x}^{(L+1)}}\left(\tau^{\epsilon}_{n-1}< t-s< \tau^{\epsilon}_n, ~Supp(X_{\tfrac{t-s}{\sigma}}^{\epsilon,\sigma})\in B^{(L+n)}\right)ds\\
  &=\int_0^t \beta(\texttt{x}^{(L)})\exp\left(-s\beta(\texttt{x}^{(L)})\right) \\
     &\qquad\quad\int_{h\in\R^d}\lim\limits_{\epsilon\to 0}\PP_{\texttt{x}^{(L+1)}}\left(\tau^{\epsilon}_{n-1}< t-s< \tau^{\epsilon}_n, ~Supp(X_{\tfrac{t-s}{\sigma}}^{\epsilon,\sigma})\in B^{(L+n)}\right)\kappa(\texttt{x}^{(L)}, dh)ds.
\eea
\eeq

By \eqref{Proof:mainTheorem_RecursionFirstStep} and \eqref{Proof:mainTheorem_RecursionAsymptoticFirstStep}, we conclude that
\beq
\lim\limits_{\epsilon\to 0}\PP_{\texttt{x}^{(L)}}
    \left(0\leq t<\tau^{\epsilon}_{1}, Supp(X_{t/\sigma}^{\epsilon,\sigma})\in B^{(L)}\right)
  =P_{\texttt{x}^{(L)}}\left( 0\leq t< S_1, Z_t\in B^{(L)}\right),
\eeq
and moreover, by \eqref{Proof:mainTheorem_Recursion} and \eqref{Proof:mainTheorem_RecursionAsymptotic}, we conclude that
\beq
\lim\limits_{\epsilon\to 0}\PP_{\texttt{x}^{(L)}}
    \left(\tau^{\epsilon}_n< t<\tau^{\epsilon}_{n+1}, Supp(X_{t/\sigma}^{\epsilon,\sigma})\in B^{(L+n)}\right)
 =P_{\texttt{x}^{(L)}}
    \left( S_n< t< S_{n+1}, Z_t\in B^{(L+n)}\right).
\eeq

Thus, by conditional probability and Lemma \ref{Lemma_new config. TST}, for any $\varepsilon >0$
\beq
\bea
\lim\limits_{\epsilon\to 0}
    &\PP_{\texttt{x}^{(L)}}
         \left(\tau^{\epsilon}_n< t<\tau^{\epsilon}_{n+1}, ~Supp(X_{t/\sigma}^{\epsilon,\sigma})\in B^{(L+n)}, ~\| X_{t/\sigma}^{\epsilon, \sigma}-\Gamma^{(L+n)}\|<\varepsilon\right)\\
    &=\lim\limits_{\epsilon\to 0}\PP_{\texttt{x}^{(L)}}
         \left(\tau^{\epsilon}_n< t<\tau^{\epsilon}_{n+1}, ~Supp(X_{t/\sigma}^{\epsilon,\sigma})\in B^{(L+n)}\right)\PP_{\texttt{x}^{(L+n)}}\left(\| X_{t/\sigma}^{\epsilon, \sigma}-\Gamma^{(L+n)}\|<\varepsilon\right)\\
    &=P_{\texttt{x}^{(L)}}\left( S_n< t< S_{n+1}, Z_t\in B^{(L+n)}\right).
\eea
\eeq

Therefore, Theorem \ref{TST_theorem_infinite trait} is proved in the sense of one-dimensional time marginal distribution. It can be generalized to convergence in finite dimensional distribution similarly as in \cite{Cha06}.
%
\end{proof}


\section{Simulation algorithm}\label{section 2.6}
The pathwise construction of the TST process defined in Definition \ref{TST_Definition_infinite tree} leads to the following numerical algorithm for simulation of the TST process.

\begin{itemize}
\item [{\it Step 0.}] Specify the  initial condition: $\Gamma_0=\Gamma^{(0)}=\bar\xi(x_0)\delta_{x_0}$.
\item [{\it Step 1.}] Simulate $\tau_1$ exponential distributed with parameter $\bar\xi(x_0)\mu(x_0)$. Sample a new trait $(x_0+h)$ with density $p(x_0, dh)$. If $f(x_0+h,x_0)>0$, relabel $x^{(1)}_0:=x_0,\,x^{(1)}_1:=x_0+h$. Otherwise, relabel $x^{(1)}_0:=x_0+h,\,x^{(1)}_1:=x_0$. \\
    Set $\Gamma^{(1)}=\bar\xi(x^{(1)}_1)\delta_{x^{(1)}_1}$, and $\Gamma_t=\Gamma^{(0)}$ for $t\in[0, \tau_1)$.
\item [{\it Step 2.}] Simulate $\tau_2$ exponential distributed with parameter $\bar\xi(x^{(1)}_1)\mu(x^{(1)}_1)$. \\
    Set $\Gamma_t=\Gamma^{(1)}$ for $t\in[\tau_1, \tau_1+\tau_2)$. \\
    Sample a new trait $\big(x^{(1)}_1+h\big)$ with density $p(x^{(1)}_1,dh)$. \\
    Choose one from the following to carry out:
    \begin{itemize}
          \item if $f(x^{(1)}_1+h,x^{(1)}_1)>0$, relabel $x^{(2)}_0:=x^{(1)}_0,\,x^{(2)}_1:=x^{(1)}_1,\,x^{(2)}_2:=x^{(1)}_1+h$;
          \item if $f(x^{(1)}_1+h,x^{(1)}_1)<0,\, f(x^{(1)}_1+h,x^{(1)}_0)>0$, relabel $x^{(2)}_0:=x^{(1)}_0,\,x^{(2)}_1:=x^{(1)}_1+h,\,x^{(2)}_2:=x^{(1)}_1$;
          \item if $f(x^{(1)}_1+h,x^{(1)}_0)<0$, relabel
              $x^{(2)}_0:=x^{(1)}_1+h,\,x^{(2)}_1:=x^{(1)}_0,\,x^{(2)}_2:=x^{(1)}_1$.
    \end{itemize}
    Set $\Gamma^{(2)}=\bar\xi(x^{(2)}_0)\delta_{x^{(2)}_0}+\bar\xi(x^{(2)}_2)\delta_{x^{(2)}_2}$.
\item [{\it Step 2l+1.}] Generate $\Gamma^{(2l+1)}$ from $\Gamma^{(2l)}=\sum\limits_{i=0}^l\bar\xi(x_{2i}^{(2l)})\delta_{x_{2i}^{(2l)}}$ for $l=1,2,\cdots$.\\
    Simulate $\tau_{2l+1}$ exponential distributed with parameter $\sum\limits_{i=0}^l \bar\xi(x_{2i}^{(2l)})\mu(x_{2i}^{(2l)})$.\\
    Set $\Gamma_t=\Gamma^{(2l)}$ for $t\in\big[\sum\limits_{i=1}^{2l}\tau_{i},\sum\limits_{i=1}^{2l+1}\tau_{i}\big)$.
    Select one trait $x^{(2l)}_{2k}$, for any $0\leq k\leq l$, to mutate with probability $\frac{\bar\xi(x_{2k}^{(2l)})\mu(x_{2k}^{(2l)})}{\sum_{i=0}^l\bar\xi(x_{2i}^{(2l)})\mu(x_{2i}^{(2l)})}$. Sample a new trait $\big(x^{(2l)}_{2k}+h\big)$ with probability density $p(x^{(2l)}_{2k}, dh)$. Choose one from the following three cases to carry out:
    \begin{itemize}
    \item if $f(x^{(2l)}_{2k}+h, x^{(2l)}_{2l})>0$, relabel $x^{(2l+1)}_{2l+1}:=x^{(2l)}_{2k}+h,\,x^{(2l+1)}_{i}:=x^{(2l)}_{i}$ for $0\leq i\leq 2l$;
    \item if $f(x^{(2l)}_{2k}+h, x^{(2l)}_0)<0$, relabel
         $x^{(2l+1)}_{i}:=x^{(2l)}_{i-1}$ for $1\leq i\leq 2l+1$, and $x^{(2l+1)}_0:=x^{(2l)}_{2k}+h$;
    \item otherwise, there exists $0\leq j<l$ s.t. $f(x^{(2l)}_{2k}+h,x^{(2l)}_{2i})<0$ for $j<i\leq l$, and $f(x^{(2l)}_{2k}+h,x^{(2l)}_{2j})>0$. Furthermore,
        \begin{itemize}
        \item if $f(x^{(2l)}_{2k}+h,x^{(2l)}_{2j+1})<0$, relabel $x_i^{(2l+1)}:=x_i^{(2l)}$ for $0\leq i\leq 2j\,$,  $x_i^{(2l+1)}:= x_{i-1}^{(2l)}$ for $2j+2\leq i\leq 2l+1$, and $x_{2j+1}^{(2l+1)}:=x_{2k}^{(2l)}+h$;
        \item if $f(x^{(2l)}_{2k}+h,x^{(2l)}_{2j+1})>0$, relabel
            $x_i^{(2l+1)}:=x_i^{(2l)}$ for $0\leq i\leq 2j+1\,$,
            $x_i^{(2l+1)}:= x_{i-1}^{(2l)}$ for $2j+3\leq i\leq 2l+1$, and
             $x_{2j+2}^{(2l+1)}:=x_{2k}^{(2l)}+h$.
        \end{itemize}
    \end{itemize}
    Set $\Gamma^{(2l+1)}=\sum\limits_{i=1}^{l+1}\bar\xi(x_{2i-1}^{(2l+1)})\delta_{x_{2i-1}^{(2l+1)}}$.
\item [{\it Step 2l+2.}] To generate $\Gamma^{(2l+2)}$
    from $\Gamma^{(2l+1)}=\sum\limits_{i=1}^{l+1}\bar\xi(x_{2i-1}^{(2l+1)})\delta_{x_{2i-1}^{(2l+1)}}$ for $l=1,2,\cdots$.\\
    This can be done as similar as the induction from $\Gamma^{(2l)}$ to $\Gamma^{(2l+1)}$. So forth.
\end{itemize}

\appendix

\section{Stability of a Lotka-Volterra system}
Consider a Lotka-Volterra system $\left(n(x), n(y)\right)$ satisfying the following equations.

\beq\label{Appendix LV2 system}
\left\{
  \begin{array}{ll}
    \dot{n}_t(x)=\left(b(x)-d(x)-\alpha(x,x)n_t(x)-\alpha(x,y)n_t(y)\right)n_t(x) \\
     \dot{n}_t(y)=\left(b(y)-d(y)-\alpha(y,x)n_t(x)-\alpha(y,y)n_t(y)\right)n_t(y).
  \end{array}
\right.
\eeq

Suppose that $n_0(x), n_0(y)>0$ and $f(y,x):=b(y)-d(y)-\al(y,x)\bar{n}(x)>0$,\, $\bar n(x)=\frac{b(x)-d(x)}{\al(x,x)}$, and its symmetric form $f(x,y)<0$. Then we conclude that $(0, \bar n(y))$ is the only stable point.

In fact, there are four fixed points of above system, namely, $(0,0)$, \,$(\bar n(x),0)$, \,$(0, \bar n(y))$, and $(n^*(x), n^*(y))$, where $(n^*(x), n^*(y))$ is such that

\beq\nonumber
\left\{
  \begin{array}{ll}
    b(x)-d(x)-\alpha(x,x)n_t(x)-\alpha(x,y)n_t(y)=0 \\
    b(y)-d(y)-\alpha(y,x)n_t(x)-\alpha(y,y)n_t(y)=0.
  \end{array}
\right.
\eeq

By simple calculation, we obtain that
\beq\nonumber
\left\{
  \begin{array}{ll}
    n^*(x)=\frac{\al(y,y)f(x,y)}{\al(x,x)\al(y,y)-\al(x,y)\al(y,x)} \\
    n^*(x)=\frac{\al(x,x)f(y,x)}{\al(x,x)\al(y,y)-\al(x,y)\al(y,x)}.
  \end{array}
\right.
\eeq
To make sense of the solution as a population density (which must be non-negative), one needs $f(x,y)\cdot f(y,x)>0$. It contradicts the assumption $f(x,y)<0, f(y,x)>0$. We thus exclude the solution $(n^*(x), n^*(y))$.

The Jacobian matrix for the system \eqref{Appendix LV2 system} at point $(0,0)$ is
\beq\nonumber
\left(
  \begin{array}{cc}
    b(x)-d(x) & 0 \\
    0 & b(y)-d(y) \\
  \end{array}
\right).
\eeq
Obviously its eigenvalues are both positive. Thus $(0,0)$ is unstable.

The Jacobian matrix at point $(\bar n(x), 0)$ is
\beq\nonumber
\bea
\left(
  \begin{array}{cc}
    -\left(b(x)-d(x)\right) & -\al(x,y)\bar n(x) \\
    0 & b(y)-d(y)-\al(y,x)\bar n(x) \\
  \end{array}
\right)\\
=\left(
   \begin{array}{cc}
      -\left(b(x)-d(x)\right) & -\al(x,y)\bar n(x) \\
      0 & f(y,x) \\
   \end{array}
 \right).
\eea
\eeq
Since one of its eigenvalue $-\left(b(x)-d(x)\right)$ is negative whereas the other one is $f(y,x)>0$, the equilibrium $(\bar n(x), 0)$ is unstable.

The Jacobian matrix of system \eqref{Appendix LV2 system} at point $(0, \bar n(y))$ is
\beq\nonumber
\left(
  \begin{array}{cc}
    b(x)-d(x)-\al(x,y)\bar n(y) & 0 \\
    -\al(y,x)\bar n(y) & -\left(b(y)-d(y)\right) \\
  \end{array}
\right)\\
=\left(
   \begin{array}{cc}
     f(x,y) & 0 \\
     -\al(y,x)\bar n(y) & -\left(b(y)-d(y)\right) \\
   \end{array}
 \right),
\eeq
whose eigenvalues are both negative because of the condition $f(x,y)<0$.
Thus $(0, \bar n(y))$ is the only stable equilibrium of the system \eqref{Appendix LV2 system}.

\bibliographystyle{abbrv}
\bibliography{bib}

\end{document}

%% file: test2.pdf_tex

\begingroup
  \makeatletter
  \providecommand\color[2][]{%
    \errmessage{(Inkscape) Color is used for the text in Inkscape, but the package 'color.sty' is not loaded}
    \renewcommand\color[2][]{}%
  }
  \providecommand\transparent[1]{%
    \errmessage{(Inkscape) Transparency is used (non-zero) for the text in Inkscape, but the package 'transparent.sty' is not loaded}
    \renewcommand\transparent[1]{}%
  }
  \providecommand\rotatebox[2]{#2}
  \ifx\svgwidth\undefined
    \setlength{\unitlength}{312.95268555pt}
  \else
    \setlength{\unitlength}{\svgwidth}
  \fi
  \global\let\svgwidth\undefined
  \makeatother
  \begin{picture}(1,0.34575122)%
    \put(0,0){\includegraphics[width=\unitlength]{test2.pdf}}%
    \put(0.01197992,0.04583982){\color[rgb]{0,0,0}\makebox(0,0)[lb]{\smash{${\scriptstyle{x_0}}$}}}%
    \put(0.01453622,0.12508508){\color[rgb]{0,0,0}\makebox(0,0)[lb]{\smash{${\scriptstyle{x_1}}$}}}%
    \put(0.01453622,0.20433034){\color[rgb]{0,0,0}\makebox(0,0)[lb]{\smash{${\scriptstyle{x_2}}$}}}%
    \put(0.42429347,0.00620002){\color[rgb]{0,0,0}\makebox(0,0)[lb]{\smash{{\tiny time}}}}%
    \put(-0.00085376,0.30717536){\color[rgb]{0,0,0}\makebox(0,0)[lb]{\smash{{\tiny trait}}}}%
    \put(0.0145362,0.27079412){\color[rgb]{0,0,0}\makebox(0,0)[lb]{\smash{${\scriptstyle{x_3}}$}}}%
  \end{picture}%
\endgroup

%% file: treeconstruction.pdf_tex

\begingroup
  \makeatletter
  \providecommand\color[2][]{%
    \errmessage{(Inkscape) Color is used for the text in Inkscape, but the package 'color.sty' is not loaded}
    \renewcommand\color[2][]{}%
  }
  \providecommand\transparent[1]{%
    \errmessage{(Inkscape) Transparency is used (non-zero) for the text in Inkscape, but the package 'transparent.sty' is not loaded}
    \renewcommand\transparent[1]{}%
  }
  \providecommand\rotatebox[2]{#2}
  \ifx\svgwidth\undefined
    \setlength{\unitlength}{334.14321289pt}
  \else
    \setlength{\unitlength}{\svgwidth}
  \fi
  \global\let\svgwidth\undefined
  \makeatother
  \begin{picture}(1,0.64583678)%
    \put(0,0){\includegraphics[width=\unitlength]{treeconstruction.pdf}}%
    \put(0.04473871,0.0431722){\color[rgb]{0,0,0}\makebox(0,0)[lb]{\smash{${\scriptstyle{x_0}}$}}}%
    \put(0.0471329,0.11260355){\color[rgb]{0,0,0}\makebox(0,0)[lb]{\smash{${\scriptstyle{x_1}}$}}}%
    \put(0.0471329,0.18682327){\color[rgb]{0,0,0}\makebox(0,0)[lb]{\smash{${\scriptstyle{x_2}}$}}}%
    \put(0.44526968,0.01083465){\color[rgb]{0,0,0}\makebox(0,0)[lb]{\smash{{\tiny time}}}}%
    \put(0.0087771,0.2879345){\color[rgb]{0,0,0}\makebox(0,0)[lb]{\smash{{\tiny trait}}}}%
    \put(0.16684197,0.43581825){\color[rgb]{0,0,0}\makebox(0,0)[lb]{\smash{${\scriptstyle{x_1}}$}}}%
    \put(0.16923614,0.48849032){\color[rgb]{0,0,0}\makebox(0,0)[lb]{\smash{${\scriptstyle{x_2}}$}}}%
    \put(0.16684195,0.52679732){\color[rgb]{0,0,0}\makebox(0,0)[lb]{\smash{${\scriptstyle{x_3}}$}}}%
    \put(0.16684195,0.56749843){\color[rgb]{0,0,0}\makebox(0,0)[lb]{\smash{${\scriptstyle{x_4}}$}}}%
    \put(0.49723971,0.08387347){\color[rgb]{0,0,0}\makebox(0,0)[lb]{\smash{${\scriptstyle{x_3}}$}}}%
    \put(0.49723971,0.12457455){\color[rgb]{0,0,0}\makebox(0,0)[lb]{\smash{${\scriptstyle{x_4}}$}}}%
    \put(0.88101101,0.00604627){\color[rgb]{0,0,0}\makebox(0,0)[lb]{\smash{{\tiny time}}}}%
    \put(0.46367202,0.29032865){\color[rgb]{0,0,0}\makebox(0,0)[lb]{\smash{{\tiny trait}}}}%
    \put(0.14763967,0.61114921){\color[rgb]{0,0,0}\makebox(0,0)[lb]{\smash{{\tiny trait}}}}%
    \put(0.16923614,0.37835782){\color[rgb]{0,0,0}\makebox(0,0)[lb]{\smash{${\scriptstyle{x_0}}$}}}%
    \put(0.49723967,0.03598976){\color[rgb]{0,0,0}\makebox(0,0)[lb]{\smash{${\scriptstyle{x_2}}$}}}%
    \put(0.32006986,0.34483925){\color[rgb]{0,0,0}\makebox(0,0)[lb]{\smash{${\scriptstyle{I_1}}$}}}%
    \put(0.40386658,0.34005089){\color[rgb]{0,0,0}\makebox(0,0)[lb]{\smash{${\scriptstyle{I_2}}$}}}%
    \put(0.49963377,0.33765701){\color[rgb]{0,0,0}\makebox(0,0)[lb]{\smash{${\scriptstyle{I_2+S_0}}$}}}%
    \put(0.18120715,0.00965374){\color[rgb]{0,0,0}\makebox(0,0)[lb]{\smash{${\scriptstyle{I_1}}$}}}%
    \put(0.25063883,0.00965406){\color[rgb]{0,0,0}\makebox(0,0)[lb]{\smash{${\scriptstyle{I_2}}$}}}%
    \put(0.32964681,0.00726005){\color[rgb]{0,0,0}\makebox(0,0)[lb]{\smash{${\scriptstyle{I_2+S_0}}$}}}%
    \put(0.67201537,0.00965406){\color[rgb]{0,0,0}\makebox(0,0)[lb]{\smash{${\scriptstyle{I_3}}$}}}%
    \put(0.7007454,0.00965424){\color[rgb]{0,0,0}\makebox(0,0)[lb]{\smash{${\scriptstyle{I_4}}$}}}%
    \put(0.734264,0.00486604){\color[rgb]{0,0,0}\makebox(0,0)[lb]{\smash{${\scriptstyle{I_4+S_2}}$}}}%
  \end{picture}%
\endgroup

%% file: TST.pdf_tex
\begingroup%
  \makeatletter%
  \providecommand\color[2][]{%
    \errmessage{(Inkscape) Color is used for the text in Inkscape, but the package 'color.sty' is not loaded}%
    \renewcommand\color[2][]{}%
  }%
  \providecommand\transparent[1]{%
    \errmessage{(Inkscape) Transparency is used (non-zero) for the text in Inkscape, but the package 'transparent.sty' is not loaded}%
    \renewcommand\transparent[1]{}%
  }%
  \providecommand\rotatebox[2]{#2}%
  \ifx\svgwidth\undefined%
    \setlength{\unitlength}{320.49599609bp}%
    \ifx\svgscale\undefined%
      \relax%
    \else%
      \setlength{\unitlength}{\unitlength * \real{\svgscale}}%
    \fi%
  \else%
    \setlength{\unitlength}{\svgwidth}%
  \fi%
  \global\let\svgwidth\undefined%
  \global\let\svgscale\undefined%
  \makeatother%
  \begin{picture}(1,0.40189195)%
    \put(0,0){\includegraphics[width=\unitlength]{TST.pdf}}%
    \put(0.01748267,0.07211014){\color[rgb]{0,0,0}\makebox(0,0)[lb]{\smash{${\scriptstyle x^{(3)}_0}$}}}%
    \put(0.01748267,0.13451341){\color[rgb]{0,0,0}\makebox(0,0)[lb]{\smash{${\scriptstyle x^{(3)}_1}$}}}%
    \put(0.01249041,0.24683931){\color[rgb]{0,0,0}\makebox(0,0)[lb]{\smash{${\scriptstyle x^{(3)}_2}$}}}%
    \put(0.01498654,0.35666908){\color[rgb]{0,0,0}\makebox(0,0)[lb]{\smash{${\scriptstyle x^{(3)}_3}$}}}%
    \put(0.4318404,0.09207919){\color[rgb]{0,0,0}\makebox(0,0)[lb]{\smash{${\scriptstyle x^{(3)}_0 (=x^{(4)}_0)}$}}}%
    \put(0.43433653,0.1494902){\color[rgb]{0,0,0}\makebox(0,0)[lb]{\smash{${\scriptstyle x^{(3)}_1 (=x^{(4)}_1)}$}}}%
    \put(0.43932879,0.20190895){\color[rgb]{0,0,0}\makebox(0,0)[lb]{\smash{${\scriptstyle x^{(3)}_1+h_1 (=x^{(4)}_2)}$}}}%
    \put(0.70391872,0.31673098){\color[rgb]{0,0,0}\makebox(0,0)[lb]{\smash{${\scriptstyle x^{(4)}_4+h_2(=x^{(5)}_4)}$}}}%
    \put(0.42919561,0.26926282){\color[rgb]{0,0,0}\makebox(0,0)[lb]{\smash{${\scriptstyle x^{(3)}_2(=x^{(4)}_3)}$}}}%
    \put(0.43683266,0.37164586){\color[rgb]{0,0,0}\makebox(0,0)[lb]{\smash{${\scriptstyle x^{(3)}_3(=x^{(4)}_4)}$}}}%
    \put(0.18956703,0.00467293){\color[rgb]{0,0,0}\makebox(0,0)[lb]{\smash{${\scriptstyle \Gamma^{(3)}}$}}}%
    \put(0.71889551,0.37913426){\color[rgb]{0,0,0}\makebox(0,0)[lb]{\smash{${\scriptstyle  x^{(4)}_4 (=x^{(5)}_5)}$}}}%
    \put(0.72388777,0.26930449){\color[rgb]{0,0,0}\makebox(0,0)[lb]{\smash{${\scriptstyle  x^{(4)}_3 (=x^{(5)}_3)}$}}}%
    \put(0.72124294,0.21185181){\color[rgb]{0,0,0}\makebox(0,0)[lb]{\smash{${\scriptstyle  x^{(4)}_2 (=x^{(5)}_2)}$}}}%
    \put(0.71889551,0.15947472){\color[rgb]{0,0,0}\makebox(0,0)[lb]{\smash{${\scriptstyle  x^{(4)}_1 (=x^{(5)}_1)}$}}}%
    \put(0.72139164,0.09707145){\color[rgb]{0,0,0}\makebox(0,0)[lb]{\smash{${\scriptstyle  x^{(4)}_0 (=x^{(5)}_0)}$}}}%
    \put(0.49659115,0.00467293){\color[rgb]{0,0,0}\makebox(0,0)[lb]{\smash{${\scriptstyle \Gamma^{(4)}}$}}}%
    \put(0.74870038,0.00467293){\color[rgb]{0,0,0}\makebox(0,0)[lb]{\smash{${\scriptstyle \Gamma^{(5)}}$}}}%
  \end{picture}%
\endgroup%

%% file: dynamics_2type.pdf_tex

\begingroup
  \makeatletter
  \providecommand\color[2][]{%
    \errmessage{(Inkscape) Color is used for the text in Inkscape, but the package 'color.sty' is not loaded}
    \renewcommand\color[2][]{}%
  }
  \providecommand\transparent[1]{%
    \errmessage{(Inkscape) Transparency is used (non-zero) for the text in Inkscape, but the package 'transparent.sty' is not loaded}
    \renewcommand\transparent[1]{}%
  }
  \providecommand\rotatebox[2]{#2}
  \ifx\svgwidth\undefined
    \setlength{\unitlength}{602.21875pt}
  \else
    \setlength{\unitlength}{\svgwidth}
  \fi
  \global\let\svgwidth\undefined
  \makeatother
  \begin{picture}(1,0.67985315)%
    \put(0,0){\includegraphics[width=\unitlength]{dynamics_2type.pdf}}%
    \put(0.1176846,0.12025583){\color[rgb]{0,0,0}\makebox(0,0)[lb]{\smash{ ${\scriptstyle {T^{\epsilon,1}}}$}}}%
    \put(0.49761299,0.12025583){\color[rgb]{0,0,0}\makebox(0,0)[lb]{\smash{${\scriptstyle {T^{\eta,1}}}$}}}%
    \put(0.60920035,0.11627056){\color[rgb]{0,0,0}\makebox(0,0)[lb]{\smash{${\scriptstyle {\widetilde{T}^{\eta,1}}}$}}}%
    \put(0.79517929,0.12025583){\color[rgb]{0,0,0}\makebox(0,0)[lb]{\smash{${\scriptstyle {T^{\epsilon,0}}}$}}}%
    \put(-0.01328421,0.66421047){\color[rgb]{0,0,0}\makebox(0,0)[lb]{\smash{{\tiny density}}}}%
    \put(0.93864875,0.13619688){\color[rgb]{0,0,0}\makebox(0,0)[lb]{\smash{\tiny time}}}%
    \put(0.33156038,0.65826631){\color[rgb]{0,0,0}\makebox(0,0)[lb]{\smash{${\scriptstyle {\xi_t^{\epsilon}(x_0)}}$}}}%
    \put(0.31960459,0.29693581){\color[rgb]{0,0,0}\makebox(0,0)[lb]{\smash{${\scriptstyle {\xi_t^{\epsilon}(x_1)}}$}}}%
    \put(0.04594987,0.2517695){\color[rgb]{0,0,0}\makebox(0,0)[lb]{\smash{${\scriptstyle\eta}$}}}%
    \put(0.03930777,0.17870635){\color[rgb]{0,0,0}\makebox(0,0)[lb]{\smash{${\scriptstyle\epsilon}$}}}%
    \put(0.07650355,0.06313372){\color[rgb]{0,0,0}\makebox(0,0)[lb]{\smash{${\scriptstyle {O(1)}}$}}}%
    \put(0.22395828,0.04984952){\color[rgb]{0,0,0}\makebox(0,0)[lb]{\smash{${\scriptstyle {f_{1,0}^{-1}\ln \frac{\eta}{\epsilon}}}$}}}%
    \put(0.54496004,0.03020471){\color[rgb]{0,0,0}\makebox(0,0)[lb]{\smash{${\scriptstyle  {O(1)}}$}}}%
    \put(0.64639614,0.00866847){\color[rgb]{0,0,0}\makebox(0,0)[lb]{\smash{${\scriptstyle\rho_1\ln \frac{\eta}{\epsilon}}$}}}%
    \put(-0.02363981,0.6090117){\color[rgb]{0,0,0}\makebox(0,0)[lb]{\smash{${\scriptstyle\bar{\xi}(x_0)}$}}}%
    \put(-0.01832612,0.54199656){\color[rgb]{0,0,0}\makebox(0,0)[lb]{\smash{${\scriptstyle\bar{\xi}(x_1)}$}}}%
  \end{picture}%
\endgroup

%% file: phase3type.pdf_tex

\begingroup
  \makeatletter
  \providecommand\color[2][]{%
    \errmessage{(Inkscape) Color is used for the text in Inkscape, but the package 'color.sty' is not loaded}
    \renewcommand\color[2][]{}%
  }
  \providecommand\transparent[1]{%
    \errmessage{(Inkscape) Transparency is used (non-zero) for the text in Inkscape, but the package 'transparent.sty' is not loaded}
    \renewcommand\transparent[1]{}%
  }
  \providecommand\rotatebox[2]{#2}
  \ifx\svgwidth\undefined
    \setlength{\unitlength}{724.6203125pt}
  \else
    \setlength{\unitlength}{\svgwidth}
  \fi
  \global\let\svgwidth\undefined
  \makeatother
  \begin{picture}(1,0.47248984)%
    \put(0,0){\includegraphics[width=\unitlength]{phase3type.pdf}}%
    \put(0.11510123,0.05002157){\color[rgb]{0,0,0}\makebox(0,0)[lb]{\smash{ ${\scriptstyle {T^{\epsilon,1}}}$}}}%
    \put(0.64282588,0.04670949){\color[rgb]{0,0,0}\makebox(0,0)[lb]{\smash{${\scriptstyle {\widetilde{T}^{\eta,2}}}$}}}%
    \put(0.82167816,0.05443767){\color[rgb]{0,0,0}\makebox(0,0)[lb]{\smash{${\scriptstyle {T^{\eta,0}}}$}}}%
    \put(-0.00036873,0.45793131){\color[rgb]{0,0,0}\makebox(0,0)[lb]{\smash{{\tiny density}}}}%
    \put(0.94753718,0.06547794){\color[rgb]{0,0,0}\makebox(0,0)[lb]{\smash{\tiny time}}}%
    \put(0.0554838,0.16152824){\color[rgb]{0,0,0}\makebox(0,0)[lb]{\smash{${\scriptstyle\eta}$}}}%
    \put(0.04996367,0.10080678){\color[rgb]{0,0,0}\makebox(0,0)[lb]{\smash{${\scriptstyle\epsilon}$}}}%
    \put(0.00427324,0.35685529){\color[rgb]{0,0,0}\makebox(0,0)[lb]{\smash{${\scriptstyle\bar{\xi}(x)}$}}}%
    \put(0.29395351,0.05002157){\color[rgb]{0,0,0}\makebox(0,0)[lb]{\smash{ ${\scriptstyle {T^{\eta,1}}}$}}}%
    \put(0.38006758,0.04560546){\color[rgb]{0,0,0}\makebox(0,0)[lb]{\smash{ ${\scriptstyle {\widetilde{T}^{\eta,1}}}$}}}%
    \put(0.55008765,0.05002157){\color[rgb]{0,0,0}\makebox(0,0)[lb]{\smash{ ${\scriptstyle {T^{\eta,2}}}$}}}%
    \put(0.69144374,0.00371745){\color[rgb]{0,0,0}\makebox(0,0)[lb]{\smash{\tiny recovery of $\scriptstyle{x_0}$}}}%
  \end{picture}%
\endgroup

%% file: phase4type.pdf_tex

\begingroup
  \makeatletter
  \providecommand\color[2][]{%
    \errmessage{(Inkscape) Color is used for the text in Inkscape, but the package 'color.sty' is not loaded}
    \renewcommand\color[2][]{}%
  }
  \providecommand\transparent[1]{%
    \errmessage{(Inkscape) Transparency is used (non-zero) for the text in Inkscape, but the package 'transparent.sty' is not loaded}
    \renewcommand\transparent[1]{}%
  }
  \providecommand\rotatebox[2]{#2}
  \ifx\svgwidth\undefined
    \setlength{\unitlength}{724.6203125pt}
  \else
    \setlength{\unitlength}{\svgwidth}
  \fi
  \global\let\svgwidth\undefined
  \makeatother
  \begin{picture}(1,0.5195089)%
    \put(0,0){\includegraphics[width=\unitlength]{phase4type.pdf}}%
    \put(0.11510123,0.05222972){\color[rgb]{0,0,0}\makebox(0,0)[lb]{\smash{ ${\scriptstyle {\widetilde T^{\eta,2}}}$}}}%
    \put(0.73777215,0.05664583){\color[rgb]{0,0,0}\makebox(0,0)[lb]{\smash{${\scriptstyle {T^{\eta,1,1}}}$}}}%
    \put(-0.00036873,0.50650854){\color[rgb]{0,0,0}\makebox(0,0)[lb]{\smash{{\tiny density}}}}%
    \put(0.94753718,0.06768609){\color[rgb]{0,0,0}\makebox(0,0)[lb]{\smash{\tiny time}}}%
    \put(0.0554838,0.16373639){\color[rgb]{0,0,0}\makebox(0,0)[lb]{\smash{${\scriptstyle\eta}$}}}%
    \put(0.04996367,0.10301494){\color[rgb]{0,0,0}\makebox(0,0)[lb]{\smash{${\scriptstyle\epsilon}$}}}%
    \put(0.00427323,0.35906343){\color[rgb]{0,0,0}\makebox(0,0)[lb]{\smash{${\scriptstyle\bar{\xi}(x)}$}}}%
    \put(0.34253068,0.05222972){\color[rgb]{0,0,0}\makebox(0,0)[lb]{\smash{ ${\scriptstyle {T^{\eta,3}}}$}}}%
    \put(0.4683897,0.04560556){\color[rgb]{0,0,0}\makebox(0,0)[lb]{\smash{ ${\scriptstyle {\widetilde{T}^{\eta,3}}}$}}}%
    \put(0.50425549,0.00371745){\color[rgb]{0,0,0}\makebox(0,0)[lb]{\smash{\tiny recovery of $\scriptstyle{x_1}$}}}%
  \end{picture}%
\endgroup

%% file: TST-1.bbl
\begin{thebibliography}{10}

\bibitem{BP97}
B.~Bolker and S.~Pacala.
\newblock Using moment equations to understand stochastically driven spatial
  pattern formation in ecological systems.
\newblock {\em Theor. Popul. Biol.}, 52:179--197, 1997.

\bibitem{BW12b}
A.~Bovier and S.~D. Wang.
\newblock Multi-time scales in adaptive dynamics: microscopic interpretation of
  a trait substitution tree model.
\newblock 2012.
\newblock Preprint.

\bibitem{Cha06}
N.~Champagnat.
\newblock A microscopic interpretation for adaptive dynamics trait substitution
  sequence models.
\newblock {\em Stoch. Proc. Appl.}, 116:1127--1160, 2006.

\bibitem{CL07}
N.~Champagnat and A.~Lambert.
\newblock Evolution of discrete populations and the canonical diffusion of
  adaptive dynamics.
\newblock {\em Ann. Appl. Probab.}, 17:102--155, 2007.

\bibitem{CM10}
N.~Champagnat and S.~M\'el\'eard.
\newblock Polymorphic evolution sequence and evolutionary branching.
\newblock {\em Probab. Theor. and Relat. Field.}, 148, 2010.

\bibitem{DG10}
D.~A. Dawson and A.~Greven.
\newblock Multiscale analysis: Fisher-wright diffusions with rare mutations and
  selection, logistic branching system.
\newblock 2010.

\bibitem{DL96}
U.~Dieckmann and R.~Law.
\newblock The dynamical theory of coevolution: a derivation from stochastic
  ecological processes.
\newblock {\em J. Math. Biol.}, 34:579--612, 1996.

\bibitem{DK99}
P.~J. Donnelly and T.~M. Kurtz.
\newblock A countable representation of the fleming-viot measure-valued
  diffusions.
\newblock {\em Ann. Probab.}, 24:698--742, 1999.

\bibitem{Eth04}
A.~M. Etheridge.
\newblock Survival and extinction in a locally regulated population.
\newblock {\em Ann. Appl. Probab.}, 14:188--214, 2004.

\bibitem{FM04}
N.~Fournier and S.~M\'el\'eard.
\newblock A microscopic probabilistic description of a locally regulated
  population and macroscopic approximation.
\newblock {\em Ann. Appl. Probab.}, 14:1880--1919, 2004.

\bibitem{GVG12}
R.~J. Gillies, D.~Verduzco, and R.~A. Gatenby.
\newblock Evolutionary dynamics of carcinogenesis and why targeted therapy does
  not work.
\newblock {\em Nature Reviews Cancer}, 12:487--491, 2012.

\bibitem{holzel2013}
M.~H{\"o}lzel, A.~Bovier, and T.~T{\"u}ting.
\newblock Plasticity of tumour and immune cells: a source of heterogeneity and
  a cause for therapy resistance?
\newblock {\em Nature Reviews Cancer}, 2013.

\bibitem{Huang11}
S.~Huang.
\newblock The molecular and mathematical basis of waddington~s epigenetic
  landscape: A framework for post-darwinian biology?
\newblock {\em Bioessays}, 34:149--157, 2011.

\bibitem{HW07}
M.~Hutzenthaler and A.~Wakolbinger.
\newblock Ergodic behavioer of locally regulated branching populations.
\newblock {\em Ann. Appl. Probab.}, 17:474--501, 2007.

\bibitem{Li10}
Z.~H. Li.
\newblock {\em Measure-valued branching Markov processes}.
\newblock Springer, 2010.

\bibitem{MT09}
S.~M\'el\'eard and V.~C. Tran.
\newblock Trait substitution sequence process and canonical equation for
  age-structured populations.
\newblock {\em Journal of Math. Biol}, 58:881--921, 2009.

\bibitem{MGMJ96}
J.~A.~J. Metz, S.~A.~H. Geritz, G.~Mesz\'ena, F.~A.~J. Jacobs, and J.~S.
  Van~Heerwaarden.
\newblock Adaptive dynamics: a geometrical study of the consesequences of
  nearly faithful reproduction.
\newblock {\em Stochastic and Spatial Structures of Dynamical System}, pages
  183--231, 1996.

\end{thebibliography}
